\documentclass[11pt,reqno]{amsart}
\usepackage[english]{babel}

\usepackage{amsmath, amssymb, amsthm, bm, esint, mathrsfs, stmaryrd, cancel, nicefrac, upgreek}
\usepackage{comment, subfiles, environ, soul, relsize, aligned-overset, array, enumerate, xparse, etoolbox}
\usepackage{extarrows}    
\usepackage{graphicx, float, rotating, wrapfig, tikz}
\DeclareGraphicsExtensions{.bmp,.png,.pdf,.jpg}
\usetikzlibrary{positioning,babel}
\usepackage{listings}
\usepackage{fancyhdr, datetime, anysize, multicol}
\usepackage{hyperref}
\usepackage{setspace}
\usepackage{enumitem}
\usepackage{titlesec}

\usepackage[paper=a4paper]{geometry}
\setlist{topsep=0.1em,itemsep=0.3em,parsep=0em,partopsep=0em}

\hypersetup{
  colorlinks=true,
  citecolor=red!80!black,
  linkcolor=blue!75!black,
  urlcolor=blue!75,
  citebordercolor=red,
  linkbordercolor=blue}

\def\om{\omega}
\def\veps{\varepsilon}
\def\vphi{\varphi}
\def\on{\operatorname}
\def\ds{\displaystyle}
\def\xlim{\xrightarrow}
\def\wlim{\rightharpoonup}

\def\supp{\on{supp}}

\def\ad{\mathrm{ad}}

\def\wt{\widetilde}
\def\wh{\widehat}
\def\ul{\underline}

\newcommand{\set}[1]{\left\{ #1\right\}}
\newcommand{\brk}[1]{\left[ #1\right]}
\newcommand{\nor}[2]{\left\|#1\right\|_{#2}}

\def\qin{\quad\text{ in }~}

\def\qan{\quad\text{ and }\quad}

\def\bbE{{\mathbb E}}

\def\N{\mathbb{N}} 
\def\bbP{\mathbb P}

\def\cA{{\mathcal A}}

\def\cE{{\mathcal E}}

\def\cH{{\mathcal H}}
\def\cI{{\mathcal I}}

\def\cR{{\mathcal R}}

\def\cS{{\mathcal S}}

\def\cZ{{\mathcal{Z}}}

\def\sJ{\mathscr J}
\def\sL{\mathscr L}

\def\bD{\mathbf D}

\def\bg{\mathbf g}

\def\bQ{\mathbf Q}
\def\bR{\mathbf R}
\def\bN{\mathbf N}

\def\bT{\mathbf T}

\def\bv{\bm v}
\def\bt{\mathbf t}
\def\bw{\bm w}
\def\bW{\mathbf W}

\def\rE{{\rm E}}

\def\rL{{\rm L}}

\def\rP{{\rm P}}

\def\rT{{\rm T}}
\def\rU{{\rm U}}

\def\rW{{\rm W}}
\def\rZ{{\rm Z}}

\def\R{\bR}
\def\N{\bN}

\def\pd{\mathrm t}  
\def\ps{\bt}        

\def\argmin{\mathop{{\mathrm{argmin}}}\nolimits}

\let\oldforall=\forall
\renewcommand\forall{\oldforall\,}

\newcommand{\spa}{\mathrm{span}}

\numberwithin{equation}{section}

\newtheoremstyle{assump} 
    {}                    
    {}                    
    {}                   
    {}                           
    {\normalfont\scshape}                   
    {.}                          
    { }                       
    {}  

\newtheorem{definition}{Definition}[section]

\newtheorem{theorem}{Theorem}[section]
\newtheorem{lemma}{Lemma}[section]
\newtheorem{prop}{Proposition}[section]
\theoremstyle{definition}
\newtheorem{remark}{Remark}[section]

\theoremstyle{assump}
\newtheorem{assump}{Assumption}[section]

\makeatletter
\def\section{\@startsection{section}{1}%
  \z@{5ex \@plus 1ex \@minus .2ex}{3ex \@plus .2ex}%
  {\Large\normalfont\scshape\centering}}

\def\subsection{\@startsection{subsection}{2}%
  \z@{2.0ex \@plus .8ex \@minus .2ex}{.8ex \@plus .2ex}%
  {\large\normalfont\bf}}

\def\subsubsection{\@startsection{subsubsection}{3}%
  \z@{1.5ex \@plus .8ex \@minus .2ex}{.5ex \@plus .2ex}%
  {\normalfont\itshape}}
\makeatother

\newcommand{\acknowledgments}[1]{%
  \refstepcounter{section}
  \addcontentsline{toc}{section}{Acknowledgments}
  \noindent\textbf{Acknowledgments.} #1
}

\title{Risk-averse optimization under distributional uncertainty with Rockafellian relaxation}
\author{Harbir Antil}
\address{Department of Mathematical Sciences and Center for Mathematics and Artificial Intelligence,
George Mason University, Fairfax, VA, 22030, USA}
\email{hantil@gmu.edu}
\thanks{This work is partially supported by the Office of Naval Research (ONR) under Award NO: N00014-24- 1-2147, the Air Force Office of Scientific Research (AFOSR) under Award NO: FA9550-25-1-0231, NSF grant DMS-2408877, and SURE-AI Centre grant 357482, Research Council of Norway.}

\author{Alonso J. Bustos}
\address{GIANuC$^2$ and Departamento de Matem\'atica y F\'isica Aplicadas, 
Universidad Cat\'olica de la Sant\'isima Concepci\'on, Casilla 297, Concepci\'on, Chile, and CI$^2$MA, Universidad de Concepci\'on, Casilla 160-C, Concepci\'on,
Chile}
\email{alonso.bustos@ucsc.cl}
\thanks{}

\author{Sean P. Carney}
\address{Department of Mathematics,
Union College, Schenectady, NY, 12308, USA}
\email{carneys@union.edu}
\thanks{}

\author{Benjamín Venegas}
\address{GIANuC$^2$ and Departamento de Matem\'atica y F\'isica Aplicadas, 
Universidad Cat\'olica de la Sant\'isima Concepci\'on, Casilla 297, Concepci\'on, Chile, and CI$^2$MA, Universidad de Concepci\'on, Casilla 160-C, Concepci\'on,
Chile}
\email{benjamin.venegas@ucsc.cl}
\thanks{}

\usepackage[textsize=tiny]{todonotes}
\definecolor{lightgray}{gray}{0.6} 

\begin{document}
\setlength{\headheight}{13.0pt}
\newpage
\begin{abstract}
A framework for risk-averse optimization problems is introduced 
that is resilient to ambiguities
in the true form of the underlying probability distribution.
The focus is on problems with partial differential equations 
(PDEs) as constraints, 
although the formulation is more broadly applicable. 
The framework is based on combining risk measures with problem relaxation techniques, 
and it builds off 
of previous advances for risk-neutral problems. 
This work advances the existing theory with strengthened
$\Gamma$-convergence results, novel existence results and 
first-order optimality criteria. 
In particular, 
the theoretical approach naturally accommodates infinite-dimensional probability 
spaces;  no finite-dimensional noise assumption is needed. 
The framework blends aspects of both 
distributionally robust optimization (DRO) and distributionally 
optimistic optimization (DOO) approaches. 
The DRO aspect facilitates strong out-of-sample performance, while the 
DOO aspect takes care of adversarial and outlier data, as illustrated 
with numerical examples.
%
\end{abstract}

\maketitle

\section{Introduction}\label{sec:introduction}
In a wide variety of applications in financial planning, operations management,
engineering design, and statistical modeling, it is critical to account for 
uncertainty in a conservative, risk-averse manner. 
Risk measures are a well-established tool in stochastic optimization to arrive 
at decisions that minimize exposure to exceptionally poor outcomes \cite{royset2025risk}.
Given a random variable $\xi$ defined on a probability space $(\Xi, \cA, \bbP)$, $1\le p\le \infty$, and a 
 coherent risk measure    
\cite[Section 6.3]{shapiro2014lectures} 
$\cR: L^p(\Xi, \cA, \bbP) \to \R$, a typical risk-averse stochastic program can 
be written in the form 
\begin{equation}\label{eq:generic}
\min_{z\in \rZ_\ad} \mathcal{R}[f(\xi,z)] .
\end{equation}
Here the deterministic control variable $z$ belongs to an admissible 
subset $\rZ_\ad$ of some Banach space $Z$, and $f: \Xi \times Z \to \R$  
is integrable for each $z \in Z$. 

In practice, the precise form of the uncertainty itself
is often unsettled; i.e.\ the sample space $\Xi$ and probability measure $\bbP$ are
themselves uncertain.
The assumed probability distribution might 
not accurately reflect the true variability, or extreme outcomes 
may appear outside the nominal support of the distribution. These inaccuracies could
arise from measurement error or adversarial corruption, for example.
Ignoring these possibilities can lead to optimization solutions 
that perform poorly under real-world conditions, particularly 
in the risk-averse setting \cite{kouri2018optimization}. 

A common approach to account for such meta-uncertainty is
to consider distributionally robust optimization (DRO) formulations,
where one minimizes an expected value over a collection $\{\bbP_i\}_{i\in \cI}$ 
of plausible measures, which is termed the ambiguity set. 
In risk-averse approaches, worst-case scenarios are
considered, which yields a minimax problem 
\begin{equation}\label{eq:generic_dro}
\min_{z\in \rZ_\ad}  \sup_{i \in \cI} \int_{\Xi} f(\xi, z) d\bbP_i . 
\end{equation}
The literature on DRO is quite vast; we refer to the reviews 
\cite{ben2009robust,postek2016computationally,rahimian2022frameworks,lin2022distributionally}
and references therein. 

It is well known that minimizing a coherent measure of risk $\cR$ is itself 
an example of a distributionally robust optimization problem. 
For example, if one takes $\cR =\on{CVaR}_{\beta}$, i.e.\ the Conditional Value-at-Risk
at level $\beta \in (0,1)$ \cite{RU,RU-2000}, then \eqref{eq:generic} becomes 
\begin{equation}\label{eq:generic_cvar}
\min_{z\in\rZ_\ad} \on{CVaR}_{\beta}(f(\xi,z)) = 
   \min_{z\in\rZ_\ad} \inf_{\gamma\in \R} \Big( \gamma + \frac{1}{1-\beta} \int_{\Xi} (f(\xi,z)-\gamma)_+ \, d\bbP\Big). 
\end{equation}
By the Fenchel-Moreau theorem \cite[Theorem 9.3.5]{ABM}, \eqref{eq:generic_cvar} is equivalent to 
\begin{equation}
\min_{z\in\rZ_\ad} \sup_{\theta \in \Theta_{\beta}} \int_{\Xi} \theta(\xi) f(\xi,z) \, d\bbP, 
\end{equation}
where the so-called risk envelope 
$$
\Theta_{\beta} = \Big\{ \vartheta \in L^{\infty}(\Xi, \cA,\bbP) 
   \, \Big\vert \, \int_{\Xi} \vartheta(\xi)\, d\bbP = 1, 
   \, 0 \le \vartheta \le (1-\beta)^{-1} \, \bbP\text{ a.s.}   \Big\}
$$ 
plays the role of the ambiguity set. 
For more general coherent risk measures, an analogous result can be found in 
\cite[Theorem 6.7]{shapiro2014lectures}.

In the context of statistical estimation, DRO is known to exhibit strong out-of-sample performance; it allows one to hedge against 
common issues such as overfitting and ``distributional shifts''---deviations  
between the distributions observed in the model training and
model deployment environments. 
We refer to the comprehensive discussion in \cite{blanchet2025distributionally} 
for more details. 
On the other hand, DRO can be ill-suited for problems with contaminated or outlier data. 
At scenarios $\xi$ that cause the function 
$f(\xi,z)$ to be exceptionally large (or possibly unbounded) at certain  
control values $z$ (termed ``endogenous'' outliers by Jiang and Xie 
\cite{jiang2024distributionally}) the DRO approach 
necessarily leads to higher objective function values, owing to the supremum 
in \eqref{eq:generic_dro}. 

An alternative approach for problems with data corruption is  
distributionally optimistic optimization (DOO). 
Generally speaking, DOO seeks best-case scenarios, for example, by replacing the supremum in \eqref{eq:generic_dro} with an infimum.  
Although the literature is not as vast as for DRO, optimistic approaches are also well-studied; 
see \cite{bi2004support,hanasusanto2017ambiguous,agarwal2020optimistic,gotoh2025data}, for example. They are also closely connected to techniques from robust statistics \cite{huber1992robust,maronna2019robust}, such as using a trimmed loss \cite{aravkin2020trimmed}, as recently pointed out by
the authors in \cite{jiang2024distributionally} and \cite{blanchet2024automatic}. This connection 
can help explain the success
of these approaches in the presence of outlier data. 
The work in \cite{ACDR} was the first example of an optimistic approach to 
stochastic optimization problems with 
partial differential equations (PDEs) as constraints.

An interesting and relatively unexplored approach to stochastic optimization 
is to combine DOO with DRO to simultaneously 
improve both resilience to corrupted data and out-of-sample performance. 
A recent example is from \cite{jiang2024distributionally}, who proposed
combining a ``favorable $\on{CVaR}_{\beta}$'' with a DRO approach based on measuring 
ambiguity sets with the Wasserstein $\infty$-distance. 
Similar approaches can be found in \cite{norton2017optimistic,nguyen2019optimistic,song2020optimistic,blanchet2024automatic}.  
Distributionally robust chance constrained programs \cite{hanasusanto2015distributionally,chen2023approximations,shen2023chance}
can also be classified in this way.

The current work presents a novel approach to combining robust and optimistic frameworks. 
The approach is based on employing both risk measures and Rockafellian relaxation techniques.
Although the framework is more general, our focus is on 
problems that are constrained by PDEs. 
In this setting, the generic stochastic optimization problem \eqref{eq:generic}
is modified to include the composition of a 
(random-valued) control-to-state map $s(\xi,z)$ with a quantity-of-interest
map $g$, as well as an additional control penalty term.

In particular, this work uses the coherent risk measure $\cR = \on{CVaR}_{\beta}$ (Eq.~\eqref{eq:generic_cvar}),
which originated in financial mathematics and portfolio optimization \cite{RU,RU-2000}. 
The authors in \cite{KS} were the first to use $\cR = \on{CVaR}_{\beta}$ in the 
context of PDE constrained optimization (PDECO). Analysis and optimization algorithms 
can be found in \cite{kouri2018existence} and \cite{kouri2022primal,garreis2021interior,beiser2023adaptive}, 
respectively. Risk measures for PDECO have also been used in conjunction with reduced order
models \cite{yang2017algorithms,heinkenschloss2018conditional} and 
tensor-train decompositions \cite{antil2023ttrisk}. Application areas include 
building and shape design \cite{kolvenbach2018approach,chaudhuri2020risk,kodakkal2022risk}, 
noise reduction \cite{yang2017algorithms}, and digital twins \cite{AAL-2025}. 
These lists are far from exhaustive, however, and risk-averse PDECO is a continually evolving field.

Rockafellian relaxation is an optimistic approach to stochastic optimization 
based on ideas originating in \cite{Rockafellar1963,Rockafellar1970}. Rather than replacing the 
supremum in \eqref{eq:generic_dro} with an infimum (or suitable reformulation thereof), 
the approach seeks best-case scenarios by not only optimizing over the control 
variable $z$, but also jointly with a choice of model perturbations \cite{royset2021good}. 
Similar ideas based on minimizing ``bivariate functions'' and ``perturbation functions'' can 
also be found in \cite{bauschke2020convex} and \cite{zalinescu2002convex}, respectively. 
Rockafellians are intimately related to Lagrangian relaxations and dual problems, as discussed
in \cite[Chapter 5]{royset2021optimization} and \cite{deride2025approximations}. 
The work in \cite{ACDR} established theoretical foundations as well as practical algorithms for Rockafellian relaxation applied to risk-neutral 
PDECO problems, building off of prior work for 
finite dimensional 
problems in \cite{royset2025rockafellian}. The numerical examples 
in \cite{ACDR} additionally showcased that 
Rockafellian relaxation enables recovery of the optimal control to an \emph{uncorrupted} 
optimization problem, even when solving with corrupted data. 
See also recent work applying Rockafellian relaxation to handle erroneous training labels in deep learning \cite{chen2024mitigating}, as well as extensions to problems with arbitrary Borel distributions and discontinuous integrands (including chance-constrained programs) \cite{tian2025approximating}.

The first of two main objectives of the present work is to introduce a number of
theoretical advances from the work in \cite{ACDR}. 
We prove novel existence results for Rockafellians constrained by PDEs, and, under 
appropriate regularity assumptions on the control-to-state map, we rigorously 
 derive first-order optimality conditions. 
Importantly, our construction allows for infinite-dimension (in the Schauder sense) probability spaces. 
In contrast to many presentations in the literature, no finite-noise 
assumption is needed, cf.~\cite{kouri2018optimization}. 
We also present a unified theory for $\Gamma$-convergence of corrupted Rockafellians 
to the original, uncorrupted one in the limit of vanishing data corruption. As is well-known, 
this guarantees preservation of convergence of minimizers, demonstrating the method's consistency.

The second main objective is to illustrate with numerical examples the practical utility
of Rockafellian relaxation for risk-averse PDECO problems with corrupted data. This builds
off of the numerical examples presented in \cite{ACDR} in the risk-neutral setting.
We observe that the method can recover the optimal control of a risk-averse problem
in the presence of endogenous outliers \cite{jiang2024distributionally} caused by corrupted
data samples. 
Additionally, we illustrate how the combined DOO and DRO approach can uncover  ``hidden'' possibilities:
using the $\on{CVaR}_{\beta}$ risk measure and minimizing over the Rockafellian decision variable allow
for the recovery of controls that \textit{would} be optimal, were the problem data different.
The principal utility of the combined approach is that it can equip decision makers with
 insight into the sensitivity of optimal control strategies under
changing levels of \textit{both} problem corruption and risk-aversion.

The remainder of this paper is structured as follows. In Section \ref{sec:notation and preliminaries}, we introduce 
some notations and preliminary definitions. Section~\ref{sec:uncorrupted problems} studies the original uncorrupted problem together with natural modifications that prepare the ground for the introduction of corruption. Section~\ref{sec:unified analysis} constitutes the main part of this work, in which we develop a systematic framework to analyze corrupted versions of the problems introduced previously. In particular, we introduce several constructions that allow us to handle arbitrary sample spaces, which may be infinite-dimensional and of infinite measure. Then, we prove the existence of minimizers, first-order optimality conditions, and the convergence of corrupted minimizers to uncorrupted minimizers. In Section \ref{sec:particular cases}, we particularize the results given in Section \ref{sec:unified analysis} to some cases of interest. In particular, as mentioned earlier, we complete the analysis made in~\cite{ACDR} for corruption to the continuous distribution and to the support of distributions. Moreover, we illustrate how focusing on particular instances enables us to relax assumptions and strengthen the obtained results.
Section~\ref{sec:numerics} contains two numerical examples of Rockafellian relaxation in practice, and Section~\ref{sec:conclusions} finishes with some concluding remarks. 

\section{Notation and preliminaries}\label{sec:notation and preliminaries}

Let $(\rZ,\|\cdot\|_{\rZ})$ be a reflexive Banach space, $(\rU,\|\cdot\|_{\rU})$ be an arbitrary Banach space, and $(\Xi,\cA,\bbP)$ be a probability space whose sample space is equipped with a norm $\nor{\cdot}{\Xi}$. Let $\rZ_{\ad}\subset \rZ$ be a closed convex subset of optimization variables. We consider functions $f_0:\rZ\to \overline\R$, $g:\rU\to \overline\R$, and $s:\Xi\times \rZ\to \rU$, which satisfy the following assumptions.
\begin{assump}[Properties of the solution map $s$]\label{as:on s}\hfill
\begin{enumerate}
\item $s(\cdot,z):\Xi\to \rU$ is $\cA$-measurable, for every $z\in \rZ$.
\item If both $\xi_\veps \to \xi$ in $\Xi$ and $z_\veps \wlim z$ in $\rZ$ as $\veps \downarrow 0$, then 
\begin{equation*}
s(\xi_\veps,z_\veps) \wlim s(\xi,z) \qin \rU\, .
\end{equation*}
\end{enumerate}
\end{assump}
\begin{assump}[Properties of $f_0$ and $g$]\label{as:on f0 and g}\hfill
    \begin{enumerate}
    \item $f_{0}$ is proper: $f_0(z) > -\infty \quad \forall z\in \rZ$ and $f_0(z) < +\infty$ for some $z\in \rZ$.
    \item Both $f_{0}$ and $g$ are sequentially weakly lower semicontinuous; that is, respectively,
    \begin{align*}
        &z_\veps \overset{\rZ}{\rightharpoonup} z \implies \liminf_{\veps \downarrow 0} f_0(z_\veps) \geq f_0(z)\, , \\
        &u_\veps \overset{\rU}{\rightharpoonup} u \implies \liminf_{\veps \downarrow 0} g(z_\veps) \geq g(z)\, . 
    \end{align*}
    \end{enumerate}
\end{assump}

Now, let us introduce some notation that will be used throughout the paper. Given an arbitrary space $X$ and a subset $S \subset X$, we denote by $\imath_S$ the indicator function of $S$, namely, the function $X \to \R$ defined by $\imath_S(x)=0$ for $x \in S$ and $\imath_S(x)=+\infty$ for $x \notin S$. Similarly, let $\chi_{S}$ denote the characteristic function of $S$, that is, the function that takes the value $1$ in $S$, and $0$ otherwise. We also denote by $\rE_{S,0}$ the extension-by-zero operator from $S$ to $X$. More precisely, given a function $f:S\to \overline{\R}$, its extension $\rE_{S,0}(f):X\to \overline{\R}$ coincides with $f$ on $S$, and is identically zero on $X \setminus S$.

We also recall the definition of the positive part function $(\, \cdot\, )_+:\R\to \R$, given by
\begin{equation*}
x_+ \,:=\, \frac{|x|+x}2 \,=\, \max\, \{0,x\}    \qquad \forall x\in \R \, ,
\end{equation*}
which enjoys the following useful lower semicontinuity property: for any sequence of real numbers $(x_{\veps})_{\veps > 0}$ such that $\ds \liminf_{\veps\downarrow0} x_{\veps}\geq x$ for some $x\in \R$, one has
\begin{equation}\label{eq:property of positive part}
\liminf_{\veps\downarrow 0} \; (x_{\veps})_+\geq x_+.
\end{equation}

We now introduce the notions of convergence that we will consider throughout the paper. In the following, $X$ and $Y$ denote Banach spaces, $\phi:X\times Y\to \R$ denotes a map, whereas $(\phi_\veps)_\veps$ is a sequence of maps from $X\times Y$ to $\R$. 
\begin{definition}[Mosco convergence] \label{def:mosco-convergence}
The sequence $(\phi_\veps)_\veps$ Mosco-converges to $\phi$ if the following conditions hold:
\begin{enumerate}
    \item For all $(x,y)\in X\times Y$ there is a sequence $(x_\veps,y_\veps)_\veps$ in $X\times Y$ converging strongly to $(x,y)$ such that
    \begin{equation}\label{eq:limsup condition-mosco}
        \limsup_{\veps\downarrow 0}\, \phi_\veps(x_\veps,y_\veps)\leq \phi(x,y)\,.
    \end{equation}
    \item For all sequences $(x_\veps,y_\veps)_\veps\wlim (x,y)$ in $X\times Y$, there holds 
    \begin{equation}\label{eq:liminf condition-mosco}
        \liminf_{\veps\downarrow 0}\, \phi_\veps(x_\veps,y_\veps)\geq \phi(x,y)\,.
    \end{equation}
\end{enumerate}
These are called limit superior and limit inferior conditions, respectively.
\end{definition}
We also introduce a slight modification of the definition of $\Gamma$-convergence, which we refer to as weak-strong $\Gamma$-convergence. This notion will be particularly useful in our context (cf.~Theorem \eqref{thm:weak-strong-gamma-convergence-result}).
\begin{definition}[Weak-strong $\Gamma$-convergence]
\label{def:Weak-strong-Gamma-convergence}
The sequence $(\phi_{\veps})_{\veps}$ weak-strong $\Gamma$-converges to $\phi$, denoted by $\phi_\veps \overset{\Gamma}{\rightharpoondown} \phi$, if the following conditions hold:
\begin{enumerate}
\item For all $(x,y)\in X\times Y$, there exists a sequence $(x_{\veps},y_\veps)_{\veps>0}$ in $X\times Y$ converging strongly to $(x,y)$ such that 
\begin{equation}\label{eq:limsup condition}
\limsup_{\veps \downarrow 0}\, \phi_{\veps}(x_{\veps},y_\veps)\leq \phi(x,y) \, .
\end{equation}
\item For all sequences $x_\veps\wlim{} x\in X$ and $y_\veps\to y\in Y$, there holds 
\begin{equation}\label{eq:liminf condition}
\liminf_{\varepsilon\downarrow 0} \, \phi_\veps(x_\veps,y_\veps)\geq \phi(x,y) \, .
\end{equation}
\end{enumerate}
\end{definition}

The main difference between this and the usual notion of $\,\Gamma$-convergence lies in the topology under which convergence is understood: here, convergence in $X$ and $Y$ is taken with respect to the weak and strong topologies, respectively, rather than the strong topology on the product space as a whole. This gives rise to a somewhat stronger notion of generalized convergence, which sits strictly between Mosco and Gamma convergence. As a consequence, and as expected, the two notions are generally not equivalent.

These notions of convergence are naturally suited to the optimization framework. Indeed, control sequences typically exhibit weak convergence. The following proposition provides a concrete illustration of these concepts.
\begin{prop}\label{prop:ws Gamma convergence preserves minimizing sequences}
Let $(x_\veps^\star,y_\veps^\star)_\veps$ be a sequence in $X\times Y$ with $x_\veps^\star\in \operatorname{argmin}\phi_\veps(x)$. Assume that the sequence $(\phi_\veps)_{\veps}$ converges to $\phi$ in one of the notions defined above, and that $(x_\veps^\star,y_\veps^\star)_\veps$ converges to $(x^\star,y^\star)\in X\times Y$ with respect to the corresponding topologies. Then, $x^\star\in\on{argmin}\phi(x)$.
\end{prop}
\begin{proof}
    Given $(x,y)\in X\times Y$, there holds $\phi_\veps(x_\veps^\star,y_\veps^\star)\leq \phi_\veps(x,y)$. Then, for any sequence $(x_\veps,y_\veps)_{\veps}$ satisfying the limit superior condition associated with $(x,y)$, we obtain
    $$
    \phi(x^\star,y^\star)\leq \liminf_\veps\,  \phi_\veps(x_\veps^\star,y_\veps^\star)
    \leq \liminf_\veps\,  \phi_\veps(x_\veps,y_\veps)\leq \limsup_\veps \, \phi_\veps(x_\veps,y_\veps)\leq \phi(x,y)\,.
    $$
    This completes the proof.
\end{proof}

To conclude this section, we recall the definition of a Rockafellian associated with an optimization problem (cf.~\cite{Rockafellar1963,Rockafellar1970}).

\begin{definition}[Rockafellian]
    For Banach spaces $X$ and $Y$, a function $\varphi:X \to \overline{\R}$, and a generic optimization problem $\min_{x \in X} \varphi(x)$, a bivariate function $\Phi: X \times Y \to \overline{\R}$ is called a Rockafellian for the problem, anchored at $\overline{y} \in Y$, if
    \begin{equation*}
    \Phi(x, \overline{y}) = \varphi(x) \qquad \forall x \in X \, .
    \end{equation*}
\end{definition}

Note that Rockafellians are not unique; for a given optimization problem, there exist infinitely many Rockafellians associated with it. This non-uniqueness provides significant flexibility, allowing one to design Rockafellians that can better accommodate approximations, perturbations, or numerical regularizations than the original problem. In particular, by carefully choosing the Rockafellian, one can simplify the analysis of convergence, stability, or sensitivity of approximate solutions.

\section{The uncorrupted problem}\label{sec:uncorrupted problems}

In this section, we study the optimization problem
\begin{equation}\label{eq:augmented problem}
\min_{(z,\gamma)\in \rZ_{\ad}\times \R}\varphi (z,\gamma)= f_{0}(z)+\gamma+\kappa\, \bbE\big[\big((g\circ s)(\cdot,z)-\gamma\big)_{+}\big]\, ,
\end{equation}
together with a slight variant incorporating a smoothing of the positive-part function. Moreover, we introduce the corresponding Rockafellian formulations for each problem, thereby setting the stage for the corruption framework developed in Section~\ref{sec:unified analysis}.

\subsection{The original optimization problem}\label{subsec:original optimization problem}

Our first result establishes the existence of a solution to~\eqref{eq:augmented problem}.

\begin{prop}\label{prop:existence-of-minimizer-original-problem}
If $\rZ_\ad$ is bounded in $\rZ$ or $f_0$ is coercive, then the problem \eqref{eq:augmented problem} has a minimizer.
\end{prop}
\begin{proof}
We first prove that $\rZ\times \R\ni(z,\gamma)\mapsto \bbE\big[\big((g\circ s)(\cdot,z)-\gamma\big)_{+}\big]$ is weakly lower semicontinuous. Let $(z_n,\gamma_n)\in \rZ \times \R$ be a sequence such that $z_n\wlim z$ in $\rZ$ and $\gamma_n\wlim \gamma$ in $\R$. Notice that since $\R$ is finite-dimensional, the latter means that $\gamma_n\to \gamma$ strongly. Now, using Assumptions~\ref{as:on s} and~\ref{as:on f0 and g}, we have that $\liminf_{n\to \infty}\big\{ (g\circ s)(\xi,z_n)-\gamma_n \big\} \geq (g\circ s)(\xi,z)-\gamma$ for a.e. $\xi \in \Xi$. Then, by~\eqref{eq:property of positive part}, we get 
\begin{equation*}
    \liminf_{n\to \infty}\big( (g\circ s)(\xi,z_n)-\gamma_n \big)_+ \geq \big( (g\circ s)(\xi,z)-\gamma \big)_+ \quad \text{for a.e.}~~\xi \in \Xi\, .
\end{equation*}
An straightforward application of Fatou's lemma implies that
\begin{equation*}
    \liminf_{n\to \infty}\bbE\big[\big( (g\circ s)(\cdot,z_n)-\gamma_n \big)_+\big] \geq \bbE\big[\big( (g\circ s)(\cdot,z)-\gamma \big)_+\big]\, ,
\end{equation*}
which means that $(z,\gamma)\mapsto \bbE\big[\big((g\circ s)(\cdot,z)-\gamma\big)_{+}\big]$ is weakly lower semicontinuous, as desired. Since $f_0$ is weakly lower semicontinuous as well (cf. Assumption~\ref{as:on f0 and g}), we conclude that $\varphi$ is weakly lower semicontinuous. Moreover, $\rZ\times \R\ni (z,\gamma)\mapsto\varphi(z,\gamma)+\imath_{\rZ_\ad}(z)$ is coercive provided that $\rZ_{\ad}$ is bounded or $f_0$ is coercive. Thus, by applying~\cite[Theorem 3.2.5]{ABM}, it follows that $\varphi$ admits a minimizer in $\rZ_\ad\times \R$, i.e., problem \eqref{eq:augmented problem} has a solution.
\end{proof}

Ahead of passing to a corrupted Rockafellian framework and intending to maintain the notation as clear as possible, we introduce a slight variation of the so-called reduced objective function. Specifically, we define $\sJ$ as the function that assigns to each $(z, \gamma) \in \rZ\times \R$ the random variable
\begin{equation*}
\sJ(z,\gamma):= g(s(\cdot,z))-\gamma\, .
\end{equation*}
For $\ps:\Xi\to\Xi$, we denote $\sJ(z,\gamma;\ps)=\sJ(z,\gamma)\circ \ps$ and $\sJ_{+}(z,\gamma;\bt)=(\sJ(z,\gamma;\bt))_{+}$. We further define $f:\rZ\times \R\to \overline{\R}$ as 
\begin{equation*}
f(z,\gamma):=f_{0}(z)+\gamma \, .
\end{equation*}
Under this notation, the objective functional may be rewritten as
\begin{equation*}
\varphi(z,\gamma)=f(z,\gamma)+\kappa\, \bbE[\sJ_{+}(z,\gamma)]\, .
\end{equation*}
Furthermore, we introduce a suitable change of measure to ease the task of formalizing corruption to the continuous probability distribution:
\begin{assump}[Change of measure]\label{as:on P}\hfill
\begin{enumerate}
\item There exists another $\sigma$-finite measure $\mu$ defined on $(\Xi,\cA)$.
\item $\bbP$ is absolutely continuous with respect to $\mu$.
\end{enumerate}
\end{assump}
From now on, we assume that $\Xi$ carries the measure $\mu$. Accordingly, all measure-related properties are to be interpreted with respect to $\mu$ instead of $\bbP$.  We let $\rho$ denote the Radon--Nikodym derivative $d\bbP/d\mu$ and $\cI$ denote the integral operator $X\mapsto \int_{\Xi}X\,d\mu$. In addition, given $\pd:\Xi\to \R$ and $\ps:\Xi\to\Xi$, we denote the translations $\wp(\pd)=\rho+\pd$ and $\om(\ps)=I+\ps$, where $I:\Xi\to \Xi$ denotes the identity map. In particular, this enables us to rewrite the objective functional as
\begin{equation*}
\varphi(z,\gamma)=f(z,\gamma)+\kappa\, \cI\brk{\sJ_{+}(z,\gamma)\, \rho}=f(z,\gamma)+\kappa\, \cI\brk{\sJ_{+}\big(z,\gamma;\om(0)\big)\, \wp(0)}.
\end{equation*}
The notation on the right-hand side of the last equality is suggestive of the form that the Rockafellian will take. In fact, for $q\in (1,+\infty)$ and $q'\in (1,+\infty)$, we introduce the reflexive Banach spaces 
\begin{equation}\label{eq:def of T and bT}
\rT:=\rL^{q}(\Xi)\qan \bT:=\rL^{q'}(\Xi;\Xi)\, ,
\end{equation}
and define the announced Rockafellian functional $\Phi :\rZ\times \R\times \rT\times\bT\to \overline \R$ by
\begin{equation}\label{eq:Rockafellian}
\Phi(z,\gamma;\pd,\ps):= \begin{cases}
f(z,\gamma) +\kappa\, \cI\brk{\sJ_{+}\big(z,\gamma;\om(\ps)\big)\wp(\pd)},&\text{ if } (\pd,\ps)=(0,0)\text{ a.s.}\, ,\\
+\infty &\text{ otherwise}\, ,
\end{cases}
\end{equation}
which is clearly anchored at $(\pd,\ps)=(0,0)$. Thus, the Rockafellian problem associated with~\eqref{eq:augmented problem} is
\begin{equation}\label{eq:Rockafellian problem}
\min\Big\{ \Phi(z,\gamma;\pd,\ps)~:~ (z,\gamma,\pd,\ps)\in \rZ_{\ad}\times \R\times \rT \times \bT\Big\} \, .
\end{equation}

\begin{prop}\label{prop:existence-of-minimizer-Rockafellian-original-problem}
If $\rZ_\ad$ is bounded in $\rZ$ or $f_0$ is coercive, then the minimization problem \eqref{eq:Rockafellian problem} has at least one solution.
\end{prop}
\begin{proof}
Let $(z^\star,\gamma^\star)\in \rZ_\ad \times \R$ be a solution of \eqref{eq:augmented problem}, whose existence is guaranteed by Proposition~\ref{prop:existence-of-minimizer-original-problem}. Since $\Phi(z,\gamma;\pd,\ps)=+\infty$ whenever $(\pd,\ps)\neq(0,0)$, it follows that
\begin{equation*}
    \Phi(z^\star,\gamma^\star;0,0) =\varphi(z^\star,\gamma^\star) \leq \varphi(z,\gamma) = \Phi(z,\gamma;0,0) \leq \Phi(z,\gamma;\pd,\ps) \qquad \forall (z,\gamma)\in \rZ_{\ad}\times \R\, ,
\end{equation*}
which means that $(z^\star,\gamma^\star;0,0)$ is a solution to~\eqref{eq:Rockafellian problem}.
\end{proof}

\subsection{The \texorpdfstring{$\delta$}{delta}-smoothed uncorrupted problem}\label{subsec:delta-smoothed-uncorrupted-problem}

The non-smoothness of the positive part function introduces non-differentiability in the integral term of \eqref{eq:augmented problem}, which poses significant challenges when deriving optimality conditions---particularly in nonconvex settings, where certain regularity assumptions on the functional are essential. To address this issue, we introduce a smoothing of the positive part function, which ensures enhanced regularity of the integral term, namely Fréchet differentiability. This smoothing strategy can be traced back, for instance, to the work of Chen and Mangasarian in~\cite{CM}. Our approach can be seen as a generalization of the smoothed primal formulation studied by Kouri and Surowiec in~\cite{KS}, as both the mapping $g$ to the quantity-of-interest and the solution operator $s$ considered here are generic. In this way, we consider a function $\zeta:\R\to \R$ satisfying the following assumption.

\begin{assump}[Properties of $\zeta$]\label{as:smoothing of the positive part} \hfill 
\begin{enumerate}[label =(\roman*)]
\item $\zeta$ is continuous and bounded on the real line.\label{as:smoothing 1}
\item $\zeta$ is nonnegative and $\int_{\R} \zeta(x)\, dx=1$.\label{as:smoothing 2}
\item $\int_{\R}\zeta(x)|x|\,dx<+\infty$.\label{as:smoothing 3}
\item Either $\int_{\R}\zeta(x)x\,dx\leq 0$ or $\int_{-\infty}^{0}\zeta(x)|x|\,dx=0$.\label{as:smoothing 4}
\item The set $\set{x\in\R~:~\zeta(x)>0}$ is connected.\label{as:smoothing 5}
\end{enumerate}
\end{assump}

Under this assumption, we define a net of smoothings of the positive part $(\, \cdot\,)_{+,\delta}:\R\to\R$, given by 
\begin{equation}\label{eq:definition of smoothing}
(x)_{+,\delta}:=\int_{-\infty}^{x} A_{\delta}(\tau) \, d\tau \, , \quad \text{where}\quad A_{\delta}(\tau):=\int_{-\infty}^{\tau}\frac{1}{\delta} \, \zeta\left(\frac \sigma\delta \right)\,d\sigma \, ,
\end{equation}
for all $\delta > 0$. The smoothing proposed by Chen and Mangasarian in~\cite{CM} serves as an example of this, since it can be understood as the one obtained by considering $\zeta(x)=e^{- x}\, (1+e^{- x})^{-2}$, which certainly satisfies our assumptions.

Through a direct computation, one verifies that $(\, \cdot\,)_{+,\delta}$ is at least twice continuously differentiable, with derivatives
\begin{equation}\label{eq:derivatives of smoothing}
    (x)_{+,\delta}' = A_{\delta}(x) \qan (x)''_{+,\delta} = \frac{1}{\delta}\, \zeta\left(\frac{x}{\delta}\right) \qquad \forall x\in \R \, .
\end{equation}

Now, we collect some important facts about the smoothing for later use.
\begin{lemma}\label{lem:props of smoothing}
Consider $\zeta$ satisfying Assumption \ref{as:smoothing of the positive part}. Then, the following statements hold:
\begin{enumerate}[label = (\roman*)]
\item $\set{(\,\cdot\, )_{+,\delta}}_\delta$ converges uniformly to $(\,\cdot\,)_+$ as $\delta \downarrow 0$, with linear order.\label{lem:smoothing 1} 
\item Given $\delta>0$, either $(\,\cdot\,)_{+,\delta}\leq (\, \cdot\, )_+$ or $(\,\cdot\,)_{+}\leq (\,\cdot\,)_{+,\delta}$. \label{lem:smoothing 2}
\item Let $(x_\delta)_\delta$ be a sequence of real numbers whose limit inferior is bounded below by $x\in \R$. Then
$$
\liminf_\delta \; (x_\delta)_{+,\delta}\geq x_+\, .
$$
\label{lem:smoothing 2.5}
\item For all $\delta>0$, $(\,\cdot\,)_{+,\delta}$ is Lipschitz continuous, with uniform Lipschitz constant~$1$. \label{lem:smoothing 3}
\end{enumerate}
\end{lemma}
\begin{proof}
To prove the first assertion, we first establish upper and lower bounds for the difference $(x)_{+,\delta} - (x)_{+}$. Given $x\in \R$, integrating  by parts \eqref{eq:definition of smoothing}, we have
\begin{equation}\label{eq:id for bounds of smoothing}
(x)_{+,\delta}=-\int_{-\infty}^{x}A_{\delta}'(\tau)\, \tau\,d\tau+A_{\delta}(x)\, x \, .
\end{equation}
By definition of $A_{\delta}$  and \eqref{eq:derivatives of smoothing}, we get
$$
(x)_{+,\delta}-(x)_{+}\leq -\delta\int_{-\infty}^{x/\delta}\sigma\, \zeta(\sigma)\,d\sigma\leq \delta\,  \Delta_{2} \, ,
$$
with $\Delta_{2}=\int_{-\infty}^{0}|\sigma|\zeta(\sigma)\,d\sigma$.  Similarly, when $x>0$, from \eqref{eq:id for bounds of smoothing}, we obtain
\begin{align*}
(x)_{+,\delta}-(x)_{+} &=-\int_{-\infty}^{x}\frac \tau\delta\, \zeta(\tau/\delta)\,d\tau+(A_{\delta}(x)-1)\, x\\
&\geq -\int_{-\infty}^{x}\frac \tau\delta\, \zeta(\tau/\delta)\,d\tau -\int_{-\infty}^{x}\frac x\delta \, \zeta(\tau/\delta)\,d\tau  \geq -\delta\int_{\R}\sigma\, \zeta(\sigma)\,d\sigma \, ,
\end{align*}
where in the last inequality we have used that $x>0$ together with the natural change of variables. When $x\leq 0$, by definition, there holds
$$
(x)_{+,\delta}-(x)_{+}\geq 0 \, .
$$
The last two inequalities give $
(x)_{+,\delta}-(x)_{+}\geq -\delta\, \Delta_{1},$
with $\Delta_{1}=\max\set{ \int_{\R}\sigma\zeta(\sigma)\,d\sigma,0 }$, thereby proving the estimate
\begin{equation}\label{eq:Delta-bound-smoothing-of-positive-part}
-\delta\, \Delta_{1}\leq (x)_{+,\delta}-(x)_{+}\leq \delta \, \Delta_{2}\, ,
\end{equation}
which, in turn, provides a uniform bound on the absolute value of the difference $(x)_{+,\delta}-(x)_\delta$. Thus, letting $\delta\to 0$ yields assertion \ref{lem:smoothing 1}. The second statement follows directly from comparing \eqref{eq:Delta-bound-smoothing-of-positive-part} with Assumption \ref{as:smoothing of the positive part}, item \ref{as:smoothing 4}. Notice that the options in this statement are mutually exclusive; otherwise, the smoothing would coincide pointwise with the positive part, which leads to a contradiction, since the zero function does not integrate to $1$.

For the third item, from the uniform bounds in \eqref{eq:Delta-bound-smoothing-of-positive-part} and \eqref{eq:property of positive part}, we obtain
$$
\liminf_{\delta} \; (x_\delta)_{+,\delta}\geq \liminf_\delta \, \big[(x_\delta)_{+}-\delta\Delta_1\big]=\liminf_\delta \;  (x_\delta)_+\geq x_+\, .
$$
Finally, the uniform Lipschitz continuity of the mapping $x \mapsto (x)_{+,\delta}$ follows from its differentiability and $\|A_{\delta}\|_{\rL^{\infty}(\R)}\leq 1$. Then, given $x,y\in \R$, we have
\begin{equation*}
    \big|(x)_{+,\delta}-(y)_{+,\delta} \big| \leq \|A_{\delta}\|_{\rL^{\infty}(x,y)}\, |x-y|\leq |x-y|\, ,
\end{equation*}
which proves the third assertion and concludes the proof.
\end{proof}
Naturally, this smoothing induces a  family of ``objective'' functionals $\set{\varphi^{\delta}}_{\delta>0}$, defined by
\begin{equation}\label{eq:smoothed objective functional}
\varphi^\delta(z,\gamma) = f(z)+\kappa\, \cI\brk{\sJ_{+,\delta}\big(z , \gamma; \omega(0)\big)\, \wp(0)},
\end{equation}
to which we refer as $\delta$-smoothed objective functionals. In turn, just as in Section \ref{subsec:original optimization problem}---and more precisely, in~\eqref{eq:Rockafellian}--- one can associate to $\varphi^{\delta}$ a functional $\Phi^\delta :\rZ\times \R\times \rT\times\bT\to \overline \R$, given by
\begin{align}\label{eq:smoothed Rockafellian}
&\Phi^\delta(z,\gamma;\pd,\ps)=\left\{\begin{array}{ll}
 f(z,\gamma)+\kappa\, \cI\brk{\sJ_{+,\delta}\big(z,\gamma;\om(\ps)\big)\, \wp(\pd)},
&\text{ if } (\pd,\ps)=0\text{ a.s.}\, ,\\
 +\infty\, , &\text{ otherwise}\, ,
\end{array}\right.
\end{align}
which we call {\sl $\delta$-smoothed Rockafellian}. It is clear that it is a Rockafellian anchored at $(\pd,\ps)=(0,0)$ for $\varphi^{\delta}$. Finally, we remark that $\varphi^\delta$ admits a minimizer in $\rZ_\ad\times \R$, which follows from the same reasoning used in the proof of Proposition~\ref{prop:existence-of-minimizer-original-problem}. Consequently, by the same argument as in Proposition~\ref{prop:existence-of-minimizer-Rockafellian-original-problem}, it follows that $\Phi^\delta$ also admits a solution.

\section{Unified analysis for corrupted Rockafellians}\label{sec:unified analysis}

This is the main section of the paper. Here, we study the problems arising from simultaneous corruptions of the probability measure~$\bbP$ and the support of its distribution. This approach is motivated by the work in~\cite{ACDR}, where the authors carry out a convergence analysis for these two types of corruptions separately in the expected-value setting. Subsequently, in Sections~\ref{subsec:corruption to distribution} and~\ref{subsec:corruption to support}, we address the corresponding separated scenarios for the CVaR functional by particularizing the results developed in the present section.

We first develop a framework that encompasses the types of corruption described above and provides the theoretical foundation for their analysis. Next, we study the existence of solutions to the corrupted problems associated with~\eqref{eq:smoothed objective functional} and~\eqref{eq:smoothed Rockafellian}. We then derive first-order optimality conditions under suitable differentiability assumptions on the operators. We end this section by establishing a generalized convergence result in the sense of weak-strong $\Gamma$-convergence (see Definition \ref{def:Weak-strong-Gamma-convergence}). 

\subsection{Preliminary constructions and the Rockafellian problem}
We begin by defining the space of probability densities
\begin{equation}\label{eq:definition-of-space-P}
    \rP:=\set{\varrho: \Xi\to \R_{+} ~\Big|~ \varrho\in \rL^\infty(\Xi)\qan \int_{\Xi}\varrho(\xi) \, d\mu(\xi) = 1},
\end{equation} 
and note that the Radon--Nikodym derivative $\rho$ introduced in the previous section belongs to this space.

The natural definition of the corrupted objective functional---by analogy with the one adopted in~\cite{ACDR}---is
\begin{equation}\label{eq:no-definition of smoothed corrupted}
f(z,\gamma)+\kappa\, \cI\brk{\sJ_{+,\delta}(z,\gamma;\eta_{\veps})\, \rho_{\veps}},
\end{equation}
where $\eta_{\veps}\in\bT$ and $\rho_\veps \in \rP$ represent a corruption map and a corrupted distribution, respectively. However, since we aim to work on a general underlying measure space $\Xi$, this notion of corrupted objective functional turns out to be slightly inadequate. Indeed, in order to establish the existence of solutions to the associated minimization problem, we require the availability of suitable compact embeddings (cf. Assumption~\ref{as:on the embedding}), which, in practice, are not attainable when~$\Xi$ is infinite-dimensional. We shall return to this issue and provide further insight into this setting in Remark~\ref{rem:insight on the setting}.

First, in order to accurately represent corruptions of $\rho$ acting in our $\delta$-smoothed functional $\varphi^{\delta}$, we consider the following assumption, which should be interpreted as ``approximability'' of $\Xi$.
\begin{assump}\label{as:approximation-property-of-Xi} There exists a sequence of measurable subspaces $\{\widetilde\Xi_\veps\}_{\veps>0}$ of $\,\Xi$ such that $\chi_{_{\widetilde{\Xi}_\veps}}(\xi)\xlongrightarrow{\veps\, \downarrow \,0}1$ for almost every $\xi\in \Xi$.
\end{assump}
According to Assumption \ref{as:on P}, $\mu$ is a $\sigma$-finite measure on $(\Xi,\cA)$, so there exists an $\cA$-measurable countable covering $\{C_n\}_{n}$ of $\Xi$, such that $C_{n}\subset C_{n+1}$ and $\mu(C_{n})<+\infty$, for all $n\in \N$. Define $\widetilde{C}_{\veps}:=C_{n_\veps}$, where $n_\veps:=\left\lceil \frac{1}{\veps} \right\rceil$ and notice that, by construction, $\widetilde{C}_{\veps}\subset \widetilde{C}_{\veps'}$ for all $\veps'\leq \veps$, which implies that $\chi_{_{\widetilde{C}_\veps}}(\xi)\xlongrightarrow{\veps\, \downarrow \,0}1$ for almost every $\xi \in \Xi$.

Now, putting $\Xi_\veps := \widetilde{\Xi}_{\veps} \cap \widetilde{C}_{\veps}$, we notice that $\Xi_\veps$ has finite measure and respects the approximation property:
\begin{equation}\label{eq:approximation property of Xi}
\lim_{\veps \downarrow 0} \chi_{_{\Xi_\veps}}(\xi)=\Big(\lim_{\veps \downarrow 0}\chi_{_{\widetilde{\Xi}_\veps}}(\xi)\Big) \, \Big(\lim_{\veps \downarrow 0}\chi_{_{\widetilde{C}_\veps}}(\xi)\Big) = 1 \quad \text{for almost every } \xi \in \Xi.
\end{equation}

From now on, we assume that each set $\Xi_\veps$ carries the corresponding restriction of the measure, $\sigma$-algebra, and norm defined in $\Xi$. More precisely, the triple $(\Xi_\veps, \cA_\veps, \mu\lvert_{\cA_\veps})$ defines a finite measure space, where $\cA_\veps := \{A \cap \Xi_\veps ~:~ A \in \cA\}$, and the norm on $\Xi_\veps$ is the restriction of $\|\cdot\|_\Xi$ to $\Xi_\veps$.

With these objects at hand, by analogy with $\rP$ and $\bT$ (cf.~\eqref{eq:definition-of-space-P} and~\eqref{eq:def of T and bT}), we define 
\begin{equation}\label{eq:definition-of-space-Pveps}
\rP_\veps:=\set{ \varrho:\Xi\to \R_+~\big|~\varrho\in \rL^\infty(\Xi),\quad \supp(\varrho)\subset\Xi_{\veps}\qan\int_{\Xi_\veps}\varrho(\xi)\,d\mu(\xi)=1 }
\end{equation}
and
\begin{equation}\label{eq:definition-of-space-bQveps}
\bQ_{\veps}=\set{\ps\in \rL^{q'}(\Xi;\Xi)~:~ \supp(\ps)\subset \Xi_\veps}\, .
\end{equation}
We note that functions in $\rP_{\veps}$ (when restricted to $\Xi_{\veps}$) constitute probability densities in $\Xi_{\veps}$, and so they can be regarded as the density of a corrupted probability distribution $\bbP_{\veps} : \cA_{\veps} \to \R$. The second space $\bQ_{\veps}$ shall represent the space of corruption maps. It is direct to see the inclusions $\rP_{\veps}\subset \rP$ and $\bQ_{\veps}\subset \bT$, so we are only refining the choices of corruption according to the new family of measure spaces.

Comparing with $\cI$, $\om$ and $\wp$, we denote  $\cI_{\veps}\brk X=\int_{\Xi_{\veps}}X$, $\wp_{\veps}(\pd)=\rho_{\veps}+\pd$ and $\omega_{\veps}( \ps)=\eta_{\veps}+ \ps$, for $X:\Xi\to \R$, $\pd:\Xi\to \R$ and $\ps:\Xi\to \Xi$, respectively. Then, we may define the smoothed-corrupted objective functional as follows: given a corrupted density $\rho_{\veps}\in \rP_{\veps}$ and a  corruption map $\eta_{\veps}\in \bQ_{\veps}$, we introduce $\varphi_\veps^\delta:\rZ\times \R \to \overline{\R}$ by
\begin{equation}\label{eq:smoothed-corrupted-objective-functional}
\varphi_{\veps}^\delta(z,\gamma)=f(z,\gamma) + \kappa\, \cI_{\veps}\big[\sJ_{+,\delta}(z,\gamma;\omega_{\veps}(0))\, \wp_{\veps}(0)\big] \, ,
\end{equation}
where $\sJ_{+,\delta}(z,\gamma;\omega_\veps(0))$ denotes $\big(\sJ(z,\gamma;\omega_\veps(0))\big)_{+,\delta}$. In practical terms, the only difference with \eqref{eq:no-definition of smoothed corrupted} is that we have restricted integration to a subdomain $\Xi_{\veps}$. From now on, we assume that we have fixed nets $(\rho_\veps)_{\veps>0}$ and $(\eta_\veps)_{\veps>0}$ of corruptions.

We thus obtain our first corrupted minimization problem to study:
\begin{equation}\label{eq:epsilon-delta problem}
\min\set{\varphi_{\veps}^{\delta}(z,\gamma)~:~ (z,\gamma)\in \rZ_{\rm ad}\times \R}.
\end{equation}
As expected, we have the following result:

\begin{prop}\label{prop:existence for smoothed-corrupted}
If $\rZ_\ad$ is bounded in $\rZ$ or $f_0$ is coercive, then~\eqref{eq:epsilon-delta problem} has at least one solution.
\end{prop}
\begin{proof}
It follows from a slight modification of the proof of Proposition~\ref{prop:existence-of-minimizer-original-problem}, taking into account the properties of the smoothing $(\,\cdot\,)_{+,\delta}$ in place of those of $(\,\cdot\,)_{+}$, and using the positivity of $\rho_\veps$.
\end{proof}
Next, to define a Rockafellian associated with $\varphi_\veps^\delta$, we must define appropriate spaces for the perturbation variables. To this end, consider nets of reflexive Banach spaces $\set{\rW_\veps}_{\veps>0}$ and $\set{\bW_\veps}_{\veps>0}$, which consist of functions $\Xi_{\veps}\to \R$ and $\Xi_{\veps}\to \Xi$, respectively.
We assume that they satisfy the following assumption.
\begin{assump}\label{as:on the embedding}
For every $\veps>0$,
\begin{enumerate}
\item $\rW_\veps$ is compactly embedded into $\rL^q(\Xi_\veps)$.
\item $\bW_{\veps}$ is compactly embedded into $\rL^{q'}(\Xi_{\veps};\Xi)$.
\end{enumerate}
\end{assump}
Due to this assumption, and recalling the definition of the extension-by-zero operator (cf. Section~\ref{sec:notation and preliminaries}), we are able to define the following spaces:
\begin{equation}
\begin{aligned}\label{eq:definition of Teps}
\rT_\veps&:=\Big\{\pd:\Xi\to \R~ \big|~\pd\lvert_{_{\Xi_\veps}}\in \rW_\veps\oplus \cS_\veps^\rho \qan \pd\lvert_{_{\Xi_{\veps}^{c}}}\in \rL^{q}(\Xi_{\veps}^{c})\Big\}\\
&= \rE_{\, \Xi_{\veps},0}(\rW_{\veps})\oplus \rE_{\, \Xi_{\veps},0}(\cS_\veps^\rho)\oplus \rE_{\, \Xi_{\veps}^{c},0}\big(\rL^{q}(\Xi_{\veps}^{c})\big) \, ,
\end{aligned}
\end{equation}
where $\cS_{\veps}^\rho := \spa\{\rho_{\veps}\lvert_{_{\Xi_\veps}}\}$, and
\begin{equation}
\begin{aligned}\label{eq:definition of bTeps}
\bT_{\veps}:&=\Big\{ \ps :\Xi\to \Xi~:~  \ps\lvert_{_{\Xi_{\veps}}}\in \bW_{\veps}\oplus\cS_\veps^\eta \qan  \ps\lvert_{\Xi_{\veps}^{c}}\in \rL^{q'}(\Xi_{\veps}^{c};\Xi)\Big\}\\
&=\rE_{\Xi_{\veps},0}(\bW_{\veps})\oplus\rE_{\, \Xi_{\veps},0}(\cS_\veps^\eta)\oplus \rE_{\Xi_{\veps}^{c}}(\rL^{q'}(\Xi_{\veps}^{c};\Xi))\, ,
\end{aligned}
\end{equation}
where $\cS_\veps^\eta := \spa\{I|_{\Xi_\veps},\eta_\veps|_{\Xi_\veps}\}$. When writing these spaces as direct sums, we are tacitly assuming that $\rW_\veps$ and $\bW_\veps$ do not contain the added terms, because otherwise, there is no need for enrichment. In practice, if the choice of spaces includes any of the added elements, one can force a direct sum by deleting the finite-dimensional spaces generated by these elements, and so the subsequent analysis is also valid.

We imbue these spaces with the norm deduced from the direct sum decompositions. More precisely, given $\pd \in \rT_{\veps}$, we define 
$$
\|\pd\|_{\rT_{\veps}}:=\nor{w}{\rW_{\veps}}+\nor{v}{\cS_\veps^\rho}+\nor{g}{\rL^{q}(\Xi_{\veps}^{c})}\, ,
$$ 
where $w\in \rW_{\veps}$, $v\in \cS_\veps^\rho$ and $g\in \rL^{q}(\Xi_{\veps}^{c})$ are the unique functions satisfying
\begin{equation*}
    \pd=\rE_{\Xi_{\veps},0}(w)+\rE_{\Xi_{\veps},0}(v)+\rE_{{\Xi_{\veps}^{c}},0}(g)\, .
\end{equation*}
Similarly, given $\bt\in \bT_\veps$, we set 
\begin{equation*}
    \|\bt\|_{\bT_\veps} := \|\bw\|_{\bW_\veps} + \|\bv\|_{\cS_\veps^\eta} + \|\bg\|_{\rL^{q'}(\Xi_\veps^c;\Xi)}\, ,
\end{equation*}
where $\bw\in \bW_\veps$, $\bv\in \cS_\veps^\eta$ and $\bg\in \rL^{q'}(\Xi_\veps^c;\Xi)$ are the unique functions satisfying 
\begin{equation*}
    \bt = \rE_{\Xi_\veps,0}(\bw) + \rE_{\Xi_\veps,0}(\bv)+ \rE_{\Xi_\veps^c,0}(\bg)\, .
\end{equation*}

We remark that we do not require a specific norm for the spaces $\cS_\veps^\rho$ and $\cS_\veps^\eta$, as in a finite-dimensional space, all norms induce equivalent topologies. 

The use of compactly embedded spaces is, in broad terms, a way to enforce stronger regularity on the admissible functions. This additional regularity is what ultimately enables the existence of minimizers for the associated Rockafellians (see Theorem~\ref{thm:existence-of-minimizers}), introduced below. Likewise, the need to enrich the formulation with corruptions will arise naturally when establishing the weak–strong $\Gamma$-convergence of the corrupted Rockafellians towards the uncorrupted one (see Theorem~\ref{thm:weak-strong-gamma-convergence-result}). An important note is that this construction preserves the compactness property of $\rW_\veps$ and $\bW_\veps$. Moreover, since $\rT_\veps$ and $\bT_\veps$ are defined as direct sums of reflexive Banach spaces with the natural norms, they are reflexive as well.

Lastly, we consider a function $\cH:\rT\times \bT\to \R$ that is weakly lower semicontinuous, coercive, satisfies $\cH(0,0)=0$, and enjoys the following property: if $\cH(\pd_{k}, \ps_{k})\xlim{~k~}0$, then $(\pd_{k}, \ps_{k})\xlim{~k~}(0,0)$ strongly. $\cH$ could be, for instance, the $\rT\times \bT$ norm or a stronger norm. Furthermore, we introduce a sequence $(\theta_{\veps})_{\veps>0}\subset \R_{+}$ of penalty parameters that satisfies $\theta_{\veps}\xlim{\veps\to 0}+\infty$.

We are finally in a position to associate a Rockafellian functional with $\varphi_{\veps}^{\delta}$. Define $\Phi_{\veps}^\delta:\rZ\times \R\times \rT \times \bT \to \overline \R$ as
\begin{equation}\label{eq:smoothed corrupted Rockafellian}
\begin{aligned}
\Phi_{\veps}^\delta\big(z,\gamma;\pd, \ps\big) &:= f(z,\gamma)+\kappa\, \cI_{\veps}\left[\sJ_{+,\delta}\big(z,\gamma;\om_{\veps}( \ps)\big)\, \wp_{\veps}(\pd)\right] + \theta_{\veps}\, \cH(\pd, \ps) \\
&\quad +\imath_{\rT_{\veps}\cap \left(\rP-\rho_{\veps}\right)}(\pd)+\imath_{\bT_\veps}( \ps)\, .
\end{aligned}
\end{equation}
We notice that $\Phi_{\veps}^\delta$ is anchored at $(\pd,\ps)=(0,0)$ for $\varphi_{\veps}^\delta$ (cf.~\eqref{eq:smoothed-corrupted-objective-functional}). Then, the smoothed-corrupted Rockafellian problem associated with \eqref{eq:epsilon-delta problem}  is 
\begin{equation}\label{eq:problem-associated-to-smoothed-corrupted-Rockafellian}
    \min \Big\{ \Phi_{\veps}^\delta\big(z,\gamma;\pd, \ps\big) : \quad (z,\gamma;\pd,\ps)\in \rZ_{\ad}\times \R\times \rT\times \bT\Big\}\, .
\end{equation}
We remark that the existence of a solution to this problem does not follow from arguments similar to those in Proposition~\ref{prop:existence-of-minimizer-Rockafellian-original-problem}, since the definition of the Rockafellian does not incorporate the $+\infty$ part. The corresponding result is established in Section~\ref{subsec:existence-of-solution-to-smoothed-corrupted}.

The following growth condition guarantees that the integral term in \eqref{eq:smoothed corrupted Rockafellian} is finite. From now on, we assume that $q$ and $q'$ are H\"older conjugates, i.e., $1/q + 1/q' = 1$.
\begin{assump}\label{as:growth-condition-on-gos}
There exists a function $\varrho:[0,+\infty)\to[0,+\infty)$ such that 
$$
    \big|(g\circ s)(\xi,z)\big|\leq \varrho(\|z\|_{\rZ})\, \|\xi\|_{\Xi} \qquad \forall (\xi,z) \in \Xi\times \rZ\, .
$$
\end{assump}

Indeed, given $(z,\gamma,\pd,\ps)\in \rZ\times \R\times \rT\times \bT$, from the Lipschitz continuity of $(\cdot)_{+,\delta}$ (see Lemma \ref{lem:props of smoothing}), it follows that 
\begin{equation*}
    0\leq \sJ_{+,\delta}(z,\gamma;\omega_\veps(\bt))\leq |\sJ(z,\gamma;\omega_\veps(\bt))|+(0)_{+,\delta} \quad \qin \Xi_\veps\, .
\end{equation*} 
Notice that Assumption~\ref{as:growth-condition-on-gos} allows us to deduce that $(g \circ s)(\omega_\veps(\bt),z) \in \rL^{q'}(\Xi_\veps)$, since $\omega_\veps(\bt) = \eta_\veps + \bt \in \bT$. Together with the finiteness of $\mu(\Xi_\veps)$, this implies that $\sJ(z,\gamma;\omega_\veps(\bt))\lvert_{_{\Xi_\veps}} \in \rL^{q'}(\Xi_\veps)$. By the same argument, $(0)_{+,\delta}$ also belongs to $\rL^{q'}(\Xi_\veps)$. Hence, the right-hand side of the above inequality lies in $\rL^{q'}(\Xi_\veps)$ and, consequently, so does $\sJ_{+,\delta}(z,\gamma;\omega_\veps(\bt))$. Finally, since $\wp_\veps(\pd)\lvert_{_{\Xi_\veps}} \in \rL^{q}(\Xi_\veps)$, an application of H\"older’s inequality yields $\sJ_{+,\delta}(z,\gamma;\omega_\veps(\bt))\, \wp_\veps(\pd)\lvert_{_{\Xi_\veps}} \in \rL^{1}(\Xi_\veps)$, as desired.

\begin{remark}\label{rem:insight on the setting}
We end this section by giving some insight into the aforementioned setting. At first glance, the framework presented herein may appear excessively complicated or overly elaborate, but we shall see that its implications make it a reasonable approach to adopt. For instance, this approach may be helpful when $\Xi$ cannot be embedded in a finite-dimensional space (e.g. a space of sequences), as it enables us to study a limiting problem written in such a space through problems posed in finite-dimensional spaces. In this case, identifying function spaces $\Xi \to \R$ and $\Xi \to \Xi$ that are compactly embedded in $\rL^q(\Xi,\R)$ and $\rL^q(\Xi;\Xi)$ is a difficult task, since classical Rellich-type compactness results depend heavily on the properties of $\Xi$, particularly on its dimension and size (see e.g.~\cite{IvanishkoKrotov,Bjorn,Romanovskii,Franchi1999,Hajasz}). As we shall see in Section \ref{subsec:existence-of-solution-to-smoothed-corrupted}, the lack of a compactness property hinders the proof of existence of a minimizer for the smoothed-corrupted problems, as the $\rL^q$-topology turns out not to be strong enough. To overcome this difficulty, we can take the sequence $\set{\Xi_\veps}_{\veps>0}$ to be finite-dimensional which, in real applications, makes it possible to find function spaces that possess the desired compactness properties (cf. Assumption~\ref{as:on the embedding}).
\end{remark}

\begin{remark}
Notice that if $\Xi$ turns out to be included in a finite-dimensional space and has finite measure, we can take every subspace $\Xi_\veps$ to be $\Xi$ itself. In this sense, our framework includes the case in which a finite-dimensional noise assumption is made (see~\cite[Part I, Section 2]{AKLR} and the references therein). We consider this particular case in Section~\ref{subsec:finite-dimensional case}.
\end{remark}

\subsection{Existence of solution}\label{subsec:existence-of-solution-to-smoothed-corrupted}

The following result concerns the existence of a minimizer for each problem associated with the smoothed-corrupted Rockafellians. 
\begin{theorem}\label{thm:existence-of-minimizers}
Fix $\veps>0$ and $\delta>0$ and suppose that  Assumptions \ref{as:on s}, \ref{as:on f0 and g}, \ref{as:on P}, \ref{as:smoothing of the positive part}, \ref{as:approximation-property-of-Xi}, \ref{as:on the embedding} and \ref{as:growth-condition-on-gos} hold. Additionally, assume that $f_0$ is coercive or $\rZ_{\ad}$ is bounded. Then, there exists $$(z^{\star},\gamma^{\star};\pd^\star,\bt^\star)\in \rZ_{\rm ad}\times \R\times \rT_{\veps} \times \bT_{\veps}$$ such that
\begin{equation*}
\Phi_{\veps}^\delta(z^{\star},\gamma^{\star};\pd^\star,\bt^\star)\leq \Phi_{\veps}^\delta(z,\gamma;\pd,\bt)\qquad \forall (z,\gamma,\pd,\bt)\in \rZ_{\rm ad}\times \R\times \rT\times \bT \, .
\end{equation*}
\end{theorem}
\begin{proof}
We begin by defining the auxiliary functional $\wt \Phi_{\veps}^\delta:\rZ\times \R\times \rT_{\veps} \times \bT_\veps \to\overline\R$ given by
\begin{equation*}
    \wt \Phi_{\veps}^\delta(z,\gamma;\pd,\ps) := \Phi_{\veps}^\delta\lvert_{\rZ\times \R\times \rT_{\veps}\times \bT_\veps}(z,\gamma;\pd,\ps)+\imath_{\rZ_{\rm ad}}(z)\, ,
\end{equation*}
which, by definition, agrees with $\Phi_{\veps}^\delta$ in $\rZ_{\rm ad}\times \R\times \rT_{\veps}\times \bT_\veps$. We further realize that, if $\wt \Phi_{\veps}^\delta$ has a minimizer $(z^{\star},\gamma^\star;\pd^{\star},\ps^\star)\in \rZ_{\rm ad}\times \R\times \rT_{\veps}\times \bT_\veps$, we have
\[
\Phi_{\veps}^\delta(z^{\star},\gamma^\star;\pd^{\star},\ps^\star)=\wt \Phi_{\veps}^\delta(z^{\star},\gamma^\star,\pd^{\star},\ps^\star)\leq \wt \Phi_{\veps}^\delta(z,\gamma,\pd,\ps)=\Phi_{\veps}^\delta(z,\gamma;\pd,\ps)\, ,
\]
for all $(z,\gamma;\pd,\ps)\in \rZ_{\rm ad}\times \R\times \rT_{\veps}\times \bT_\veps$. Moreover, since $\Phi_{\veps}^\delta(z,\gamma;\pd,\ps)=+\infty$ whenever $(\pd,\ps)\not \in \rT_{\veps}\times \bT_\veps$, it holds
\[
\Phi_{\veps}^\delta(z^{\star},\gamma^\star;\pd^{\star},\ps^\star)\leq \Phi_{\veps}^\delta(z,\gamma;\pd,\ps) \qquad\forall (z,\gamma;\pd,\ps)\in \rZ_{\rm ad}\times \R\times \rT \times \bT\, .
\]
Thus, it suffices to show the existence of such a minimizer for $\wt \Phi_{\veps}^\delta$ in order to conclude the proof. We do so in what follows.

We start by proving that $\wt \Phi_\veps^\delta$  is weakly lower semicontinuous; in doing so, we pay special attention to the integral term. Consider sequences $(z_n)_{n\in \N}$ in $\rZ$, $(\gamma_n)_{n\in \N}$ in $\R$, $(\pd_n)_{n\in \N}$ in $\rT_\veps$, and $(\ps_n)_{n\in \N}$ in $\bT_\veps$, converging weakly to $z\in \rZ$, $\gamma\in \R$, $\pd\in \rT_\veps$, and $\bt\in \bT_\veps$, respectively. If $(z_n,\pd_n)_{n\in \N}$ does not have a subsequence contained in $\rZ_{\rm ad}\times(\rP-\rho_\veps)$, at least one of the indicator terms in the auxiliary Rockafellian equals $+\infty$, and the property follows. Otherwise, there exists a (not relabeled) subsequence  such that $z_n\in \rZ_{\rm ad}$ and $\rho_\veps+\pd_n\in \rP$, for all $n\in \N$. We may assume that the limit inferior $\ds \liminf_{n}\; \Phi_\veps^\delta(z_n,\gamma_n;\pd_n,\ps_n)$ is attained along this subsequence.

Notice that $\big(\pd_n|_{\Xi_{\veps}}\big)_{n\in \N}$ is weakly convergent in $\rW_\veps\oplus \cS_\veps^\rho$. Since the latter is compactly embedded in $\rL^q(\Xi_\veps)$, strong convergence (up to subsequences) implies almost everywhere pointwise convergence in $\Xi_\veps$. Hence, there exists a subsequence $(\pd_{n'})_{n'}$ satisfying
\begin{equation}\label{eq:pointwise a.e. convergence}
\rho_\veps(\xi)+\pd_{n'}(\xi) \xlim{n'\to \infty}\rho_\veps(\xi)+\pd(\xi)~~\text{ for a.e.}~~\xi \in \Xi_\veps\, .
\end{equation}
Similarly, since $\ps_{n'} \wlim \ps$ in $\bT_\veps$, using the compact embedding, we are able to find a subsequence (not relabeled) pointwise convergent to $\ps$. Then,
\begin{equation*}
    \eta_\veps(\xi) + \ps_{n'}(\xi) \xlim{n'\to \infty} \eta_\veps(\xi) + \ps(\xi)~~\text{ for a.e.}~~\xi \in \Xi_\veps\, ,
\end{equation*}
which combined with the fact that $z_n \wlim z$ in $\rZ$, and employing Assumption~\ref{as:on s}, yields
\begin{equation*}
    s(\eta_\veps(\xi)+\ps_{n'}(\xi),z_{n'}) \wlim s(\eta_\veps(\xi)+\ps(\xi),z) \qin \rU\, , ~~\text{ for a.e.}~~\xi \in \Xi_\veps\, .
\end{equation*}
Since $g$ is sequentially weakly lower semicontinuous (cf. Assumption~\ref{as:on f0 and g}), the above convergence implies that 
\begin{equation*}
    \liminf_{n'\to \infty}~ \Big( (g\circ s)(\eta_\veps(\xi)+\ps_{n'}(\xi),z_{n'}) \Big) \geq (g\circ s)(\eta_\veps(\xi)+\ps(\xi),z) ~~\text{ for a.e.}~~\xi \in \Xi_\veps\, .
\end{equation*}
Using that $\gamma_{n'}\to \gamma$ in $\R$, the monotonicity and continuity of $(\,\cdot\,)_{+,\delta}$, and the foregoing inequality, we arrive at
\begin{equation}\label{eq:wlsc of gos}
    \liminf_{n'\to \infty}~ \big(\sJ_{+,\delta}(z_{n'},\gamma_{n'};\omega_\veps(\ps_{n'}))\big)(\xi) \geq  \big(\sJ_{+,\delta}(z,\gamma;\omega_\veps(\ps))\big)(\xi)~~\text{ for a.e.}~~\xi \in \Xi_\veps \, .
\end{equation}
Bearing in mind \eqref{eq:pointwise a.e. convergence} and \eqref{eq:wlsc of gos}, and using the fact that $\rho_\veps + \pd_{n'}\geq 0$, a straightforward application of Fatou's Lemma yields
\begin{equation*}
\liminf_{n'\to \infty} ~\cI_\veps\big( \sJ_{+,\delta}\big(z_{n'},\gamma_{n'};\omega_\veps(\ps_{n'}) \big)\, \wp_{\veps}(\pd_{n'}) \big) \geq \cI_\veps\big( \sJ_{+,\delta}\big(z,\gamma;\omega_\veps(\ps) \big)\, \wp_{\veps}(\pd) \big) \, .
\end{equation*}
This, together with the weak lower semicontinuity of $f_0$ (see Assumption~\ref{as:on f0 and g}) and $\cH$, yields
$$
\liminf_{n\to \infty}~ \wt \Phi_{\veps}^\delta(z_{n},\gamma_{n};\pd_{n},\ps_n)=\liminf_{n'\to \infty}~\wt\Phi_{\veps}^\delta(z_{n'},\gamma_{n'};\pd_{n'},\ps_{n'})\geq \wt\Phi_{\veps}^\delta(z,\gamma;\pd,\ps) \, ,
$$
so $\wt\Phi_\veps^\delta$ is weakly lower semicontinuous. In turn, it is straightforward to verify that $\wt\Phi_\veps^\delta$ is coercive, since $\cH$ is coercive in $\rT \times \bT$ and either $f$ is coercive in $\rZ \times \R$ or the admissible set $\rZ_{\ad}$ is bounded. Consequently, noting that $\rZ\times \R\times\rT_\veps \times \bT_\veps$ is a reflexive Banach space, the weak lower semicontinuity and coercivity of $\wt\Phi_\veps^\delta$ allow us to invoke a particular case of the Weierstrass minimization theorem (see~\cite[Theorem 3.2.5]{ABM}) to establish the existence of a minimizer. Finally, since $\wt \Phi_\veps^\delta$ is not identically equal to $+\infty$, any minimizer must satisfy $(z^\star,\gamma^\star,\pd^\star,\ps^\star) \in \rZ_{\rm ad} \times \R \times \rT_\veps\times \bT_\veps$. According to the initial remark, this completes the proof.
\end{proof}

\begin{remark}\label{rem:alternative for compactness}
We finish this section by commenting that, as seen in the previous proof, instead of introducing Assumption \ref{as:on the embedding}, one might incorporate compact operators into the objective functional. This would allow us to dispense with the assumption that there is an actual embedding into the spaces $\rL^q(\Xi_\veps)$ and $\rL^q(\Xi_\veps;\Xi)$. We do not adopt this approach because, as far as we see, it would heavily modify the analysis of convergence to the original problem. So far, we see two possible options: First, also including compact operators in the original Rockafellian, as arguments of $\omega$ and $\wp$, would help the convergence analysis and would not modify the fact that we have Rockafellians anchored at $0$. The second option is only including them in the corrupted Rockafellians, and then deducing suitable conditions under which the results can still be carried out. One possible advantage of considering compact operators instead of an embedding is that it may allow us to strengthen the convergence results, from weak-strong Gamma convergence to Mosco convergence. This claim comes from the fact that compact operators upgrade weak convergence to strong convergence. In our setting, we are using this fact to leverage weak convergence in the proof  of Theorem \ref{thm:existence-of-minimizers}, but we are not able to  take advantage of this fact when proving convergence (see proof of Theorem \ref{thm:weak-strong-gamma-convergence-result}). If this could be overcome, one gets Mosco convergence, which is the strongest among these types of generalized convergence.
\end{remark}

\subsection{Optimality conditions}\label{subsec:optimality-conditions}

Having established the existence of minimizers for our perturbed Rockafellian problems, we now aim to derive first-order optimality conditions. We use the following notation for differentials: given an operator $A:X\to Y$ between Banach spaces, we denote its directional derivative at $x \in X$ in the direction $\ul{x} \in X$ by $\bD A(x)[\ul x]$. Furthermore, if $A$ is Fréchet differentiable, we denote its Fréchet derivative at $x$ by $\bD A(x) \in \sL(X,Y)$. When $X$ is a product space, namely $X = \prod_{i=1}^\ell X_i$, we denote by $\bD_{x_i} A(x_1,\ldots,x_\ell)[\ul x_i]$ the derivative with respect to the variable $x_i$. Notice that, if $x = (x_1,\ldots,x_\ell)$,
\begin{equation*}
    \bD A(x)[\ul x] = \bD_{x_1} A(x)[\ul x_1] + \bD_{x_2} A(x)[\ul x_2] + \cdots + \bD_{x_\ell} A(x)[\ul x_\ell]\, .
\end{equation*}

Given a convex set $S$ of a vector space $X$, the normal cone of $S$ at $x \in S$ is denoted by $N_{S}(x)$. Moreover, we use the following convention for adjoint notation: given vector normed spaces $X$, $Y$, and $Z$, and given an operator $G \in \sL(X,Y)$, we define $G^*:\sL(Y,Z)\to \sL(X,Z)$ by $G^*(F):=F\circ G$ for all $F\in \sL(Y,Z)$. In particular, for $Z=\R$ and $F\in Y'=\sL(Y,\R)$, this reduces to the standard adjoint $G^*:Y'\to X'$. Finally, we denote by $\langle \cdot,\cdot\rangle$ the duality pairing of a space and its topological dual, omitting the underlying spaces when clear from the context.

Now, notice that~\eqref{eq:problem-associated-to-smoothed-corrupted-Rockafellian} can be rewritten equivalently as
\begin{equation}\label{eq:auxiliary problem for optimality conditions}
    \min \Big\{ \wh\Phi_\veps^\delta(z,\gamma;\pd,\ps) + \imath_{\rZ_\ad \times \R\times (\rT_\veps\cap(\rP-\rho_\veps))\times \bT_\veps}(z,\gamma;\pd,\ps)~:\quad (z,\gamma;\pd,\ps)\in \rZ\times \R\times \rT \times \bT \Big\}\, ,
\end{equation}
where $\wh \Phi_\veps^\delta:\rZ\times \R\times \rT\times \bT\to \R$ is defined by
\begin{equation}\label{eq:auxiliar functional for optimality conditions}
    \wh\Phi_\veps^\delta(z,\gamma;\pd,\ps) := f(z,\gamma) + \kappa\, \cI_\veps\big( \sJ_{+,\delta}(z,\gamma;\omega_\veps(\ps))\, \wp_\veps(\pd) \big) + \theta_\veps\, \cH(\pd,\ps)\, .
\end{equation}

We use this auxiliary functional to derive the optimality conditions. As usual, we begin by establishing further assumptions on the differentiability of the functions.
\begin{assump}[Differentiability]\label{as:on-differentiability}\hfill
\begin{enumerate}
    \item $\cH:\rT\times \bT\to \R$ is Fréchet differentiable, with derivative $\bD\cH(\pd,\ps)\in \sL(\rT\times \bT,\R)$.
    \item There exists $p\in [1,+\infty]$ such that for all $\ps\in \bT$ and $z\in \rZ_\ad$, $s(\omega_\veps(\ps),z)\in \rL^p(\Xi;\rU)$.
    \item There exists an open set $V$ of $\rZ$, with $\rZ_\ad\subset V$, such that $f_0$ is Fréchet differentiable in $V$, with derivative $\bD f_0(z)\in \sL(\rZ,\R)$, for all $z\in V$.
    \item For each $z\in \rZ_\ad$, the function 
    \begin{equation*}
        \bT\ni \ps \mapsto s(\ps,z)\in \rL^p(\Xi,\rU)
    \end{equation*}
    is Fréchet differentiable with derivative $\bD_{\ps} s(\ps,z) \in \sL(\bT,\rL^p(\Xi;\rU))$.
    \item For each $\ps\in \bT$, the function
    \begin{equation*}
        V\ni z \mapsto s(\ps,z) \in \rL^p(\Xi;\rU)
    \end{equation*}
    is Fréchet differentiable with derivative $\bD_{z}s(\ps,z)\in \sL(\rZ,\rL^p(\Xi;\rU))$.
    \item The mapping 
    \begin{equation*}
        \rL^p(\Xi;\rU) \ni u \mapsto g\circ u \in \rL^{q'}(\Xi) 
    \end{equation*}
    is Fréchet differentiable with derivative $\bD g(u)\in \sL(\rL^p(\Xi;\rU),\rL^{q'}(\Xi))$.
\end{enumerate}
\end{assump}

Under Assumption~\ref{as:on-differentiability}, we are able to prove the following result:
\begin{lemma}\label{lem:derivatives-of-Rockafellian}
The functional $\wh\Phi_\veps^\delta$ is Fréchet differentiable, and its derivatives with respect to each variable are given explicitly by
\begin{equation*}
\begin{aligned}
\bD_{z}\wh \Phi_\veps^\delta(z,\gamma;\pd,\ps) &= \bD f_0(z)  + \kappa \, \cI_\veps\Big[ \wp_\veps(\pd)\, A_\delta\big( \sJ(z,\gamma;\omega_\veps(\ps)) \big)\, \bD g (s(\omega_\veps(\ps),z))\circ \bD_{z} s(\omega_\veps(\ps),z) \Big]\, , \\
\bD_{\gamma}\wh \Phi_\veps^\delta(z,\gamma;\pd,\ps) &= 1 - \kappa \, \cI_\veps\Big[ \wp_\veps(\pd)\, A_\delta\big( \sJ(z,\gamma;\omega_\veps(\ps))\big) \Big] \, , \\
\bD_{\pd}\wh \Phi_\veps^\delta(z,\gamma;\pd,\ps) &=\theta_\veps\, \bD_{\pd}\cH(\pd,\ps) + \kappa\, \cI_\veps \Big[ \sJ_{+,\delta}(z,\gamma;\omega_\veps(\ps))\, I \Big]\, ,
\end{aligned}
\end{equation*}
where $I:\bT\to \bT$ is the identity map, and
\begin{equation*}
    \begin{aligned}
        &\bD_{\ps}\wh \Phi_\veps^\delta(z,\gamma;\pd,\ps) \\
        &\qquad = \theta_\veps\, \bD_{\ps}\cH(\pd,\ps) + \kappa\, \cI_\veps\Big[ \wp_\veps(\pd)\, A_\delta\big( \sJ(z,\gamma;\omega_\veps(\ps)) \big)\, \bD g (s(\omega_\veps(\ps),z))\circ \bD_{\ps} s(\omega_\veps(\ps),z) \Big]\, .
    \end{aligned}
\end{equation*}
\end{lemma}
\begin{proof}
    This follows from the assumptions, common properties of directional derivatives (such as the product and chain rules), and straightforward algebraic manipulations. The computations are omitted for the sake of brevity.
\end{proof}

\begin{theorem}\label{thm:optimality-conditions}
Suppose that the assumptions of Theorem~\ref{thm:existence-of-minimizers}  and Lemma~\ref{lem:derivatives-of-Rockafellian} hold. If $\big(z_\veps^\star,\gamma_\veps^\star;\pd_\veps^\star,\ps_\veps^\star\big)\in \rZ_{\rm ad}\times \R \times \rT_\veps\cap (\rP-\rho_\veps) \times \bT_\veps$ is an optimal solution of the problem \eqref{eq:problem-associated-to-smoothed-corrupted-Rockafellian}, and $\lambda_\veps^\star:=(\omega_\veps(\ps_\veps^\star),z_\veps^\star)$. Then,
\begin{enumerate}
\item \label{eq:first-condition} For all $z\in \rZ_\ad$,
\begin{equation*}
\Big\langle \bD f_0(z_\veps^\star) + \kappa\, \cI_\veps \big[ \wp_\veps(t_\veps^\star)\, A_\delta\big(\sJ(z_\veps^\star,\gamma_\veps^\star;\omega_\veps(\ps_\veps^\star))\big)\, \bD_z s(\lambda_\veps^\star)^*\, \bD g(s(\lambda_\veps^\star)) \big] , z-z_\veps^\star \Big\rangle \geq 0\, .
\end{equation*}
\item \label{eq:second-condition} It holds
\begin{equation*}
\cI_\veps\Big[ \wp_\veps(\pd_\veps^\star)\, A_\delta\big( \sJ(z_\veps^\star,\gamma_\veps^\star;\omega_\veps(\ps_\veps^\star))\big) \Big]  = \frac{1}{\kappa}\, .
\end{equation*}
\item \label{eq:third-condition} For all $\pd \in \rT_\veps \cap (\rP-\rho_\veps)$,
\begin{equation*}
\Big\langle \theta_\veps\, \bD_{\pd}\cH(\pd_\veps^\star,\ps_\veps^\star) + \kappa\, \cI_\veps \Big[ \sJ_{+,\delta}(z_\veps^\star,\gamma_\veps^\star;\omega_\veps(\ps_\veps^\star))\, I \Big], \pd-\pd_\veps^\star\Big\rangle \geq 0 \, .
\end{equation*}
\item \label{eq:fourth-condition} For all $\ps \in \bT_\veps$, 
\begin{equation*}
\Big\langle \theta_\veps\, \bD_{\ps}\cH(\pd_\veps^\star,\ps_\veps^\star) + \kappa \, \cI_\veps\Big[ \wp_\veps(\pd_\veps^\star)\, A_\delta\big( \sJ(z_\veps^\star,\gamma_\veps^\star;\omega_\veps(\ps_\veps^\star)) \big)\, \bD_{\ps} s(\lambda_\veps^\star)^*\, \bD g (s(\lambda_\veps^\star) \Big], \ps-\ps_\veps^\star\Big\rangle  \geq 0 \, .
\end{equation*}
\end{enumerate}
\end{theorem}

\begin{proof}
To avoid overloading the notation, we set $\cZ_\veps^\star:=(z_\veps^\star,\gamma_\veps^\star;\pd_\veps^\star,\ps_\veps^\star)$. By~\cite[Corollary 1.12.3.]{Kr}, it follows that
\begin{equation*}
- \bD \widehat\Phi_{\veps}^\delta (\cZ_\veps^\star)\in \partial \, \imath_{\rZ_{\rm{ad}}\times \R \times \rT_\veps \cap (\rP -\rho_\veps)\times \bT_\veps}\big(\cZ_\veps^\star \big)\, ,
\end{equation*}
where the latter denotes, as usual, the subdifferential. Moreover, noting that the set in the indicator function is convex, we have 
\begin{equation*}
\partial \, \imath_{\rZ_{\rm{ad}}\times \R \times \rT_\veps \cap (\rP -\rho_\veps)\times \bT_\veps}\big(\cZ_\veps^\star \big) \cong N_{\rZ_\ad}(z_\veps^\star) \times \{0\} \times N_{\rT_\veps \cap (\rP - \rho_\veps)}(\pd_\veps^\star) \times N_{\bT_\veps}(\ps_\veps^\star)\, .
\end{equation*}
Then, 
\begin{equation*}
\begin{array}{c}
-\bD_z\widehat\Phi_{\veps}^\delta (\cZ_\veps^\star) \in N_{\rZ_\ad}(z_\veps^\star) \, , \qquad-\bD_\gamma \widehat\Phi_{\veps}^\delta (\cZ_\veps^\star) = 0\, , \\[1.5ex]
-\bD_\pd \widehat\Phi_{\veps}^\delta (\cZ_\veps^\star) \in N_{\rT_\veps \cap (\rP-\rho_\veps)}(\pd_\veps^\star) \qan -\bD_\ps \widehat\Phi_{\veps}^\delta (\cZ_\veps^\star) \in N_{\bT_\veps}(\ps_\veps^\star) \, ,
\end{array}
\end{equation*}
which mean, respectively,
\begin{equation*}
\begin{array}{c}
\big\langle \bD_z\widehat\Phi_{\veps}^\delta (\cZ_\veps^\star),z-z_\veps^\star\big\rangle \geq 0 \, , \qquad \bD_\gamma \widehat\Phi_{\veps}^\delta (\cZ_\veps^\star) = 0\, ,  \\[1.5ex]
\big\langle \bD_\pd \widehat\Phi_{\veps}^\delta (\cZ_\veps^\star) , \pd-\pd_\veps^\star \big\rangle \geq 0 \qan \big\langle \bD_\ps \widehat\Phi_{\veps}^\delta (\cZ_\veps^\star) , \ps - \ps_\veps^\star \big\rangle \geq 0 \, .
\end{array}
\end{equation*}
for all $z\in \rZ_\ad$, $\pd \in \rT_\veps \cap (\rP-\rho_\veps)$, and $\ps \in \bT_\veps$. Thus, by replacing each row according to the expressions in~Lemma \ref{lem:derivatives-of-Rockafellian} and employing the adjoint notation, we obtain the desired statements. Further details are omitted for brevity.
\end{proof}

\subsection{Weak-strong Gamma convergence}\label{sec:weak-strong-Gamma-convergence}

The following result concerns the convergence of our family of smoothed-corrupted Rockafellians to the original one, in the sense of weak-strong Gamma convergence, under a reasonable assumption on the behavior of corruption. This generalized notion of convergence is of interest because it preserves minimizing sequences, in the sense given in Proposition~\ref{prop:ws Gamma convergence preserves minimizing sequences}. 

In addition, it is worth mentioning that the following convergence theorem does not rely on existence nor optimality conditions of solutions, and so its hypothesis are milder than those of Theorems~\ref{thm:existence-of-minimizers} and~\ref{thm:optimality-conditions}

\begin{theorem}\label{thm:weak-strong-gamma-convergence-result}Suppose that Assumptions \ref{as:on s}, \ref{as:on f0 and g}, \ref{as:on P}, \ref{as:smoothing of the positive part} and \ref{as:approximation-property-of-Xi} hold. Let $\rho\in \rP\cap \rT_\veps$ for all $\veps > 0$ and consider corruptions $(\rho_{\veps})_{\veps>0}$ and $(\eta_\veps)_{\veps>0}$ such that $\rho_\veps \in \rP_{\veps}\cap\rT_{\veps}$ and $\eta_\veps\in \bQ_\veps\cap \bT_\veps$, for each $\veps>0$. Assume that 
\begin{equation}\label{eq:penalty parameters diverge}
    \lim_{\veps\downarrow 0} \, \theta_\veps=+\infty\, ,
\end{equation}
and
\begin{equation}\label{eq:assumption on corruptions}
\lim_{\veps\downarrow 0} \, \theta_\veps\cH(\rho-\rho_\veps,\eta_\veps -I)=0\, .
\end{equation}
Then, there holds $\Phi_\veps^\delta \overset{\Gamma}{\rightharpoondown} \Phi^\delta$ as $\veps \downarrow 0$, in the sense of Definition \ref{def:Weak-strong-Gamma-convergence}.
\end{theorem}
\begin{proof}
We start by proving the limit inferior condition (cf.~\eqref{eq:liminf condition}). We do so in the same spirit in which the weak lower semicontinuity of $\Phi_\veps^\delta$ was proved in Theorem \ref{thm:existence-of-minimizers}, though with the key difference that now, every object is indexed by $\veps$.  To this end, consider an arbitrary sequence $\big(z_{\veps},\gamma_{\veps},\pd_{\veps},\ps_\veps\big)_{\veps>0}\subset \rZ\times \R\times \rT\times \bT$ satisfying 
$$
(z_{\veps},\gamma_{\veps})\wlim(z,\gamma)\qan (\pd_{\veps},\ps_\veps)\longrightarrow (\pd,\ps)~~\text{ as }~~\veps\downarrow 0\, .
$$
Let $\veps'$ denote the index of a subsequence along which the relevant limit inferior (cf. Definition \ref{def:Weak-strong-Gamma-convergence}) is attained. By the same reasoning made earlier, we may assume that such subsequence satisfies
\begin{equation*}
\pd_{\veps'}\in \rT_{\veps'} \cap (\rP-\rho_{\veps'})\qan
\ps_{\veps '}\in \bT_{\veps'}\, ,
\end{equation*}
so that the indicator terms in $\Phi_\veps^\delta$ vanish. In addition, we notice for later use that, from \eqref{eq:penalty parameters diverge} and \eqref{eq:assumption on corruptions}, one gets convergence
$$
\rho_{\veps'} \xlim{~\veps'~}\rho \qin \rT \qan \eta_{\veps'}\xlim{\veps'}I \qin \bT \, .
$$
This, together with the weak convergence of $(z_{\veps'})_{\veps'}$ to $z$ and the properties of $s$,  gives
$$
s(\eta_{\veps'}+\ps_{\veps'},z_{\veps'})\wlim s(I+\ps,z)\qin \rU,~~\text{a.e. in }\Xi \, .
$$
Furthermore, using the weak convergence of $(\gamma_{\veps '})_{\veps '}$, weak lower semicontinuity of $g$, and the continuity and monotonicity of $(\,\cdot\,)_{+,\delta}$, respectively, we obtain that
\begin{equation}\label{eq:lim inf bound of J}
\liminf_{\veps'}\sJ_{+,\delta}(z_{\veps'},\gamma_{\veps'};\om_{\veps'}(\ps_{\veps'}))\geq \sJ_{+,\delta}(z,\gamma;\om(\ps)) \, . 
\end{equation}
From the fact that convergence in $\rT$ implies almost everywhere pointwise convergence, and since $\pd_{\veps'} \in \rP - \rho_{\veps'}$, it follows that
$$
\rho_{\veps'}+\pd_{\veps'}\xlim{~\veps'~}\rho+\pd\qan \rho+\pd \geq 0 ~~\text{ a.e. in }\Xi \, .
$$
Moreover, according to the approximation property of the subspaces $\Xi_\veps$ (cf.~\eqref{eq:approximation property of Xi}), we get
$$
\chi_{_{\Xi_{\veps'}}}(\rho_{\veps'}+\pd_{\veps'})\xlim{~\veps'~}\rho+\pd ~~\text{ a.e. in }\Xi \, .
$$
This, together with~\eqref{eq:lim inf bound of J} and an application of Fatou’s lemma, yields
$$
\liminf_{\veps'} \, \cI_{\veps'}\big[ \sJ_{+,\delta}(z_{\veps'},\gamma_{\veps'};\om_{\veps'}(\ps_{\veps'}))\, \wp_{\veps'}(\pd) \big]\geq \cI[\sJ_{+,\delta}(z,\gamma;\om(\ps))\, \wp(\pd)] \, .
$$
Next, we treat the penalty term. We distinguish two cases: when $(\pd,\ps)=(0,0)$ and when $(\pd,\ps)\neq (0,0)$. In this way, we notice that
$$
\liminf_{\veps'}\, \theta_\veps \, \cH(\pd_{\veps '},\ps_{\veps '})\geq 
\begin{cases}
0 &\text{if } (\pd,\ps)=(0,0)\, ,  \\
+\infty &\text{otherwise}\, , 
\end{cases}
$$
where the first case follows from the nonnegativity of $\cH$ and the penalty parameters, and the second one follows from the properties of the limit inferior. Therefore, the limit inferior condition follows from this, the weak lower semicontinuity of $f_0$, and the fact that we are working in the minimizing sequence indexed by $\veps'$.

In second place, we prove the limit superior condition (cf.~\eqref{eq:limsup condition}). Fixed $(z,\gamma,\pd,\ps)\in \rZ\times \R\times \rT\times \bT$, we must find a sequence converging to it, satisfying the bound in \eqref{eq:limsup condition}. We distinguish two cases. First, when $(\pd,\ps)=(0,0)$ a.e. in $\Xi$, we choose the sequence $(z_\veps,\gamma_\veps,\pd_\veps,\ps_\veps)_\veps=(z,\gamma,\rho-\rho_\veps,I-\eta_\veps)_\veps$, which is clearly convergent. Under this choice, we have
$$
\Phi_\veps^\delta(z_\veps,\gamma_\veps;\pd_\veps,\ps_\veps)=f(z,\gamma)+\cI_\veps\big[\sJ_{+,\delta}(z,\gamma;\om(\pd))\wp(\ps)\big] \, .
$$
We remark that the fact that the indicator terms vanish comes precisely from the enrichment of the spaces (see \eqref{eq:definition of Teps} and \eqref{eq:definition of bTeps}) and the assumption that $\rho \in \rT_\veps$. If this was not the case, the desired bound would not hold. The approximation property of $\Xi_\veps$ and the {\sl Lebesgue dominated convergence theorem} yield
$$
\cI_\veps\big[\sJ_{+,\delta}(z,\gamma;\om(\pd))\, \wp(\ps)\big]\xlim{~\veps~}\cI\big[\sJ_{+,\delta}(z,\gamma;\om(\pd))\, \wp(\ps)\big]\, ,
$$
where we used that, by definition, $\cI_\veps(Y)=\cI(\chi_{_{\Xi_\veps}}Y)$, for any $Y:\Xi\to \R$. It follows from the above arguments that
$$
\Phi_\veps^\delta(z_\veps,\gamma_\veps;\pd_\veps,\ps_\veps) \xlim{~\veps~} \Phi^\delta (z,\gamma;\pd,\ps)\, .
$$
The remaining case is when $(\pd,\ps)$ is not essentially zero. According to the definition of $\Phi^\delta$, the condition is tautological. This concludes the proof.
\end{proof}
We end this section by proving that the smoothing does not interfere with the convergence. Furthermore, we are able to prove Mosco convergence to the original Rockafellian functional.
\begin{theorem}\label{thm:convergence of smoothed to original Rockafellian}
Suppose that Assumptions \ref{as:on s}, \ref{as:on f0 and g}, \ref{as:on P} and \ref{as:smoothing of the positive part} hold. Then, the sequence of smoothed Rockafellians $\set{\Phi^\delta}_\delta$ Mosco converges to $\Phi$, as $\delta\to 0$, in the sense of Definition~\ref{def:mosco-convergence}.
\end{theorem}
\begin{proof}
First, to prove the limit inferior condition (cf.~\eqref{eq:liminf condition-mosco}), given a sequence $(z_\delta,\gamma_\delta,\pd_\delta,\ps_\delta)_\delta$ weakly converging to $(z,\gamma,\pd,\ps)$ in $\rZ\times \R\times \rT\times \bT$, we distinguish between two cases: If $(\pd,\ps)$ is not zero, no weakly convergent sequence $(\pd_\delta,\ps_\delta)_\delta$ can be identically zero. Therefore, up to deleting finite terms, we may assume that any sequence weakly converging to $(\pd,\ps)$ does not have zero terms. A simple comparison of the involved Rockafellians (cf.~\eqref{eq:Rockafellian} and~\eqref{eq:smoothed Rockafellian}) gives the limit inferior condition in this case. On the other hand, if $(\pd,\ps)$ is essentially zero, we may assume that any weakly convergent sequence $(\pd_\delta,\ps_\delta)_\delta$, up to deleting finite terms, is identically zero, because otherwise the condition is trivially verified. Then, by recalling the definition of the Rockafellians (cf.~\eqref{eq:Rockafellian} and~\eqref{eq:smoothed Rockafellian}), we have
$$
\Phi^\delta(z_\delta,\gamma_\delta;0,0)=f(z_\delta,\gamma_\delta)+\kappa \, \cI\brk{\sJ_{+,\delta}(z_\delta,\gamma_\delta,\om(0)) \, \wp(0)} \, ,
$$
and
$$
\Phi(z,\gamma;0,0)=f(z,\gamma) + \kappa\, \cI\brk{\sJ_{+}(z,\gamma,\om(0))\, \wp(0)}\, .
$$
In this way, resorting to the reasoning of the previous proof (cf. Theorem~\ref{thm:weak-strong-gamma-convergence-result}), we may establish that
$$
\liminf_\delta \sJ(z_\delta,\gamma_\delta,\om(0))\, \wp(0)\geq \sJ(z,\gamma,\om(0))\,\wp(0) \quad \text{a.e. in }\Xi\, .
$$
Moreover, according to  item \ref{lem:smoothing 2.5} in Lemma \ref{lem:props of smoothing}, the foregoing inequality implies that
$$
\liminf_\delta \sJ_{+,\delta}(z_\delta,\gamma_\delta,\om(0))\, \wp(0)\geq \sJ_{+}(z,\gamma,\om(0))\, \wp(0) \quad \text{a.e. in }\Xi\, .
$$
Just as in the previous results, a straightforward application of Fatou's lemma and the continuity properties of $f$ (cf. Assumption~\ref{as:on f0 and g}) give the condition.

We now turn to the verification of the limit superior condition (cf.~\eqref{eq:limsup condition-mosco}). The case $(\pd,\ps)\neq (0,\mathbf{0})$ is immediate. The remaining case follows by considering the constant sequence $\{(z,\gamma,\pd,\ps)\}_\delta$ and invoking the uniform upper bound provided in~\eqref{eq:Delta-bound-smoothing-of-positive-part} to control the corresponding limit superior. We omit further details for the sake of brevity. This completes the proof.
\end{proof}

\section{Particular Cases}\label{sec:particular cases}
In this section, we explore different cases that arise in applications. The analysis developed in Section~\ref{sec:unified analysis} is adapted to each context, leading to various simplifications. Specifically, we first present the modifications of the framework when the sample space $\Xi$ is a finite-dimensional space with finite measure. We then consider corruptions of continuous probability distributions, and finally address the scenario in which only the support of the probability distribution is corrupted.

\subsection{Finite-dimensional underlying measure space of finite measure}
\label{subsec:finite-dimensional case}

In this subsection, we focus on the particular case in which the sample space $\Xi$ is finite-dimensional and $\mu(\Xi)<+\infty$. This setting provides a simple framework where several technical difficulties encountered in the general case (cf. Section~\ref{sec:unified analysis}) disappear. In fact, as previously observed, in this situation we may consider the subspaces $\wt\Xi_\veps=\Xi$ in Assumption~\ref{as:approximation-property-of-Xi} and, since $\Xi$ has finite measure, we can put $\Xi_\veps = \Xi$ for all $\veps > 0$. This gives us a simplified version of the spaces introduced in Section~\ref{sec:unified analysis}. Namely,
\begin{equation*}
\rP_\veps = \rP \qan  \bQ_\veps = \bT\, .
\end{equation*}
In this way, noting that $\cI_\veps=\cI$, the smoothed-corrupted objective functional~\eqref{eq:smoothed-corrupted-objective-functional} reads:
\begin{equation*}
\varphi_\veps^\delta (z,\gamma) = f(z,\gamma) + \kappa\, \cI[\sJ_{+,\delta}(z,\gamma;\omega_\veps(0))\, \wp_\veps(0)]\, ,
\end{equation*}
for all $\veps >0$ and $\delta >0$, where $\{\rho_\veps\}_{\veps>0}\in \rP$ and $\{\eta_{\veps}\}_{\veps >0}\in \bT$ represent the corruption densities and the corruption maps, respectively. Notice that, according to Proposition~\ref{prop:existence for smoothed-corrupted}, this functional has a minimizer. 

We continue by describing how  Assumption~\ref{as:on the embedding} is simplified. First, we observe that there are Banach space isomorphisms
\begin{equation}\label{eq:simplifications-compact-embeddings-1}
[\rL^q(\Xi)]^n\xlim{\,\sim\,}\rL^q(\Xi;\Xi)\qan [\rL^{q'}(\Xi)]^n\xlim{\,\sim\,}\rL^{q'}(\Xi;\Xi) \, ,
\end{equation}
where $n:=\dim \, \Xi$. If $q'\geq q$, we only need to assume that there is a function space $\wt \rW$ with compact inclusion in $\rL^q(\Xi)$ (cf. Assumption \ref{as:on the embedding}). In fact, by finite-dimensionality, there is a space $\wt \bW$ of functions $\Xi\to \Xi$, isomorphic to $[\wt \rW]^n$. We then have 
\begin{equation}\label{eq:simplifications-compact-embeddings-2}
\wt \bW\xlim{\sim}[\wt W]^n\hookrightarrow[\rL^q(\Xi)]^n \, .
\end{equation}
Then, according to~\eqref{eq:simplifications-compact-embeddings-1},~\eqref{eq:simplifications-compact-embeddings-2}, and the canonical inclusion $\rL^{q'}(\Xi)\hookrightarrow \rL^q(\Xi)$, the choice
$$
\rW=\wt\rW \qan \bW=\wt\bW
$$
satisfies Assumption \ref{as:on the embedding}. The case where $q\geq q'$ is analogous. In the remaining of this section, we assume that the spaces $\rW$ and $\bW$ have been obtained through this construction, thereby dispensing with Assumption \ref{as:on the embedding}. Consequently, $\rT_\veps$ and $\bT_\veps$ (cf.~\eqref{eq:definition of Teps} and~\eqref{eq:definition of bTeps}) become
\begin{equation*}
\rT_\veps = \rW \oplus \spa\{\rho_\veps\} \qan \bT_\veps = \bW\oplus \spa\{I,\eta_\veps\}\, .
\end{equation*}
Therefore, the Rockafellian functional defined in~\eqref{eq:smoothed corrupted Rockafellian} is rewritten as
\begin{equation*}
\begin{aligned}
\Phi_\veps^\delta(z,\gamma;\pd,\ps) = &\;f(z,\gamma) + \kappa\, \cI[\sJ_{+,\delta}(z,\gamma;\omega_\veps(\ps))\, \wp_\veps(\pd)] + \theta_\veps \, \cH(\pd,\ps)  \\
&+ \imath_{\rT_\veps\cap(\rP-\rho_\veps)}(\pd) + \imath_{\bT_\veps}(\ps)\, .
\end{aligned}
\end{equation*}
Of course, the results presented in Section~\ref{sec:unified analysis} remain valid, with the corresponding modifications to the notation. In particular, the existence of a solution (cf. Theorem~\ref{thm:existence-of-minimizers}), the optimality conditions (cf. Theorem~\ref{thm:optimality-conditions}), and weak-strong $\Gamma$-convergence (cf. Theorem~\ref{thm:weak-strong-gamma-convergence-result}) still hold.

\subsection{Corruption to continuous probability distibution}\label{subsec:corruption to distribution}

Now, we aim to reduce the general setting presented in Section~\ref{sec:unified analysis} to the case where we only consider corruptions to continuous probability distributions. In this scenario, we maintain the space of probability densities $\rP$ defined in~\eqref{eq:definition-of-space-P} with respect to the measure $\mu$. Furthermore, in contrast to the case studied in Section~\ref{subsec:finite-dimensional case}, we do have to use the auxiliary spaces $\Xi_\veps$ (cf. Assumption~\ref{as:approximation-property-of-Xi}) so that we can construct the spaces $\rP_\veps$ (cf.~\eqref{eq:definition-of-space-Pveps}) which represent the spaces of corrupted probability densities in $\Xi_\veps$. However, $\bQ_\veps$ (cf.~\eqref{eq:definition-of-space-bQveps}) is not required anymore, as corruptions to the support of the probability distribution are not considered here. In this way, for each $\veps >0$, we let $\rho_\veps \in \rP_\veps$ be a corrupted density. Then, the smoothed-corrupted objective functional (cf.~\eqref{eq:smoothed-corrupted-objective-functional}) reads
\begin{equation*}
\varphi_{\veps}^\delta(z,\gamma)=f(z,\gamma) + \kappa\, \cI_{\veps}\big[\sJ_{+,\delta}(z,\gamma) \, \wp_{\veps}(0)\big] \, ,
\end{equation*}
where we recall that $\sJ_{+,\delta}(z,\gamma)=\sJ_{+,\delta}(z,\gamma;I)$. To construct the Rockafellian, we note from Assumption~\ref{as:on the embedding} that $\bW_\veps$ is no longer required, and it suffices to consider $\rW_\veps$. This is reflected in the spaces $\rT_\veps$ and $\bT_\veps$, where the former captures corruptions to continuous probability densities, while the latter represents corruptions to the support of the probability. Now, regarding the function $\cH:\rT \times \bT \to \R$, we consider a slight variant $\cH_{\rT}:\rT\to \R$ that enjoys similar properties. More precisely, we assume that $\cH_\rT$ is weakly lower semicontinuous, coercive, vanishes at $0$ and satisfies: if $\cH_{\rT}(t_k)\xrightarrow{k\to \infty}0$, then $t_{k}\xrightarrow{k\to \infty} 0$ in $\rT$. In this case, as mentioned when $\cH$ was introduced, one may take $\cH_{\rT}$ to be the $\rT$-norm. Thus, the Rockafellian functional associated with $\varphi_\veps^\delta$ is given by $\Phi_\veps^\delta:\rZ \times \R \times \rT \to \overline \R$, defined as
\begin{equation}\label{eq:rockafellian-corruption-to-distributions}
\Phi_{\veps}^\delta\big(z,\gamma;\pd\big) := f(z,\gamma)+\kappa\, \cI_{\veps}\left[\sJ_{+,\delta}(z,\gamma)\, \wp_{\veps}(\pd)\right] + \theta_{\veps}\, \cH_{\rT}(\pd) +\imath_{\rT_{\veps}\cap \left(\rP-\rho_{\veps}\right)}(\pd)\, .
\end{equation}
Note also that Assumption~\ref{as:growth-condition-on-gos} ensures that the integral in this functional is finite. Finally, we remark that the full Assumption~\ref{as:on s} is not required. In fact, it suffices to consider the following hypothesis.
\begin{assump}\label{as:on-s-but-weaker} \hfill
\begin{enumerate}
\item $s(\cdot,z):\Xi\to \rU$ is $\cA$-measurable, for every $z\in \rZ$.
\item If $z_\veps \wlim z$ in $\rZ$ as $\veps \downarrow 0$, then 
\begin{equation*}
s(\xi,z_\veps) \wlim s(\xi,z) \qin \rU\,, \; \text{a.s. in }\, \Xi \, .
\end{equation*}
\end{enumerate}
\end{assump}
Notice that we have only modified the second item, fixing the first component of $s$. This relaxation arises from the fact that we are now considering only corruptions to the continuous distribution, rather than both types of corruptions, so there are no sequences in the component associated with $\Xi$.

The following result establishes the existence of a minimizer for $\Phi_\veps^\delta$.
\begin{theorem}\label{thm:existence-of-minimizers-corruption-to-distribution}
Fix $\veps>0$ and $\delta>0$, and suppose that Assumptions~\ref{as:on-s-but-weaker},~\ref{as:on f0 and g},~\ref{as:on P},~\ref{as:smoothing of the positive part},~\ref{as:approximation-property-of-Xi}, (1) in~\ref{as:on the embedding}, and~\ref{as:growth-condition-on-gos} hold. Additionally, suppose that $f_0$ is coercive or $\rZ_{\ad}$ is bounded. Then, there exists $$(z^{\star},\gamma^{\star};\pd^\star)\in \rZ_{\rm ad}\times \R\times \rT_{\veps}$$ such that
\begin{equation*}
\Phi_{\veps}^\delta(z^{\star},\gamma^{\star};\pd^\star)\leq \Phi_{\veps}^\delta(z,\gamma;\pd)\qquad \forall (z,\gamma,\pd)\in \rZ_{\rm ad}\times \R\times \rT\, .
\end{equation*}
\end{theorem}
\begin{proof}
It follows from a simplification of the arguments used in the proof of Theorem~\ref{thm:existence-of-minimizers}.
\end{proof}
Next, we focus on the optimality conditions associated with the problem of minimization of $\Phi_\veps^\delta$ defined in~\eqref{eq:rockafellian-corruption-to-distributions}. To that end, we consider the same notations as in Section~\ref{subsec:optimality-conditions} and specify the hypotheses of differentiability (similar to those in Assumption~\ref{as:on-differentiability}).
\begin{assump}\label{as:on-differentiability-corruption-to-distribution} \hfill
\begin{enumerate}
\item $\cH_{\rT}:\rT \to \R$ is Fréchet differentiable, with derivative $\bD\cH_{\rT}(\pd)\in \sL(\rT,\R)$.
\item There exists $p\in [1,+\infty]$ such that for all $z\in \rZ_\ad$, $s(\cdot,z)\in \rL^p(\Xi;\rU)$.
\item There exists an open set $V$ of $\rZ$, with $\rZ_\ad\subset V$, such that $f_0$ is Fréchet differentiable in $V$, with derivative $\bD f_0(z)\in \sL(\rZ,\R)$, for all $z\in V$.
\item The function
\begin{equation*}
V\ni z \mapsto s(\cdot,z) \in \rL^p(\Xi;\rU)
\end{equation*}
is Fréchet differentiable with derivative $\bD_{z}s(\cdot,z)\in \sL(\rZ,\rL^p(\Xi;\rU))$.
\item The mapping 
\begin{equation*}
\rL^p(\Xi;\rU) \ni u \mapsto g\circ u \in \rL^{q'}(\Xi) 
\end{equation*}
is Fréchet differentiable with derivative $\bD g(u)\in \sL(\rL^p(\Xi;\rU),\rL^{q'}(\Xi))$.
\end{enumerate}
\end{assump}

\begin{theorem}
Assume the same hypotheses as in Theorem~\ref{thm:existence-of-minimizers-corruption-to-distribution}. Furthermore, suppose that Assumption~\ref{as:on-differentiability-corruption-to-distribution} holds. Let $\big(z_\veps^\star,\gamma_\veps^\star;\pd_\veps^\star\big)\in \rZ_{\rm ad}\times \R \times \rT_\veps\cap (\rP-\rho_\veps)$ be an optimal solution of the problem of minimization associated with $\Phi_\veps^\delta$. Then,
\begin{enumerate}
\item For all $z\in \rZ_\ad$,
\begin{equation*}
\Big\langle \bD f_0(z_\veps^\star) + \kappa\, \cI_\veps \big[ \wp_\veps(t_\veps^\star)\, A_\delta\big(\sJ(z_\veps^\star,\gamma_\veps^\star)\big)\, \bD_z s(\cdot,z_\veps^\star)^*\, \bD g(s(\cdot,z_\veps^\star)) \big] , z-z_\veps^\star \Big\rangle \geq 0\, .
\end{equation*}
\item It holds
\begin{equation*}
\cI_\veps\Big[ \wp_\veps(\pd_\veps^\star)\, A_\delta\big( \sJ(z_\veps^\star,\gamma_\veps^\star)\big) \Big]  = \frac{1}{\kappa}\, .
\end{equation*}
\item For all $\pd \in \rT_\veps \cap (\rP-\rho_\veps)$,
\begin{equation*}
\Big\langle \theta_\veps\, \bD_{\pd}\cH_{\rT}(\pd_\veps^\star) + \kappa\, \cI_\veps \Big[ \sJ_{+,\delta}(z_\veps^\star,\gamma_\veps^\star)\, I \Big], \pd-\pd_\veps^\star\Big\rangle \geq 0 \, .
\end{equation*}
\end{enumerate}
\end{theorem}
\begin{proof}
The proof is analogous to that of Lemma~\ref{lem:derivatives-of-Rockafellian} and Theorem~\ref{thm:optimality-conditions}, and therefore is omitted.
\end{proof}

\begin{theorem}
Let $\rho\in \rP\cap \rT_\veps$ for all $\veps > 0$ and consider corruptions $(\rho_{\veps})_{\veps>0}$ such that $\rho_\veps \in \rP_{\veps}\cap\rT_{\veps}$ for each $\veps>0$. Assume that 
\begin{equation*}
\lim_{\veps\downarrow 0} \, \theta_\veps=+\infty \qan \lim_{\veps\downarrow 0} \, \theta_\veps\cH_{\rT}(\rho-\rho_\veps)=0\, .
\end{equation*}
Then, under the same assumptions as in Theorem~\ref{thm:existence-of-minimizers-corruption-to-distribution}, there holds $\Phi_\veps^\delta \overset{\Gamma}{\rightharpoondown} \Phi^\delta$ as $\veps \downarrow 0$, in the sense of Definition \ref{def:Weak-strong-Gamma-convergence}.
\end{theorem}
\begin{proof}
The result follows by the same reasoning as in the proof of Theorem~\ref{thm:weak-strong-gamma-convergence-result}, using instead the Rockafellian defined in~\eqref{eq:rockafellian-corruption-to-distributions}.
\end{proof}

\subsection{Corruption to support of the probability distibution}\label{subsec:corruption to support}
As a last instance of analysis, we study the case in which we are only in presence of corruption to the support of the idealized probability distribution $\rho$. In this situation, the spaces $\rP_\veps$ (cf. \eqref{eq:definition-of-space-Pveps}), $\rT$, $\rW$ (cf. Assumption \ref{as:on the embedding}) and $\rT_\veps$ are no longer required, as all of them regard corruption to the probability distribution. In this way, recalling that $\wp(0)=\rho$, the objective functional (cf. \eqref{eq:smoothed-corrupted-objective-functional}) takes the form
$$
\vphi_\veps^\delta(z,\gamma) = f(z,\gamma) + \kappa \, \cI_\veps[\sJ_{+,\delta}(z,\gamma;\om_\veps(0))\, \rho] \, .
$$

In addition, we observe that the function $\cH$ now depends only on $\bT$, thus the hypotheses on it are accordingly simplified. More precisely, we require a map $\cH_\bT:\bT\to\R$ which is weakly lower semicontinuous, coercive, vanishes at $0$, and satisfies: if $\cH_\bT(\ps_k)\xlim{k}0$, then $\ps_k\xlim{k}0$ in $\bT$. In essence, these are the same requirements as before, and one may again take it to be an appropriate norm on $\bT$.

The Rockafellian functional $\Phi_\veps^\delta:\rZ\times \R \times \bT\to \R$ now reads (cf. \eqref{eq:smoothed corrupted Rockafellian})
\begin{equation}\label{eq:Rockafellian with corruption to support}
\Phi_\veps^\delta(z,\gamma;\ps)=f(z,\gamma)+\kappa \, \cI_\veps[\sJ_{+,\delta}(z,\gamma;\om_\veps(\ps))\, \rho]+\theta_\veps \, \cH_\bT(\ps)+ \imath_{\bT_\veps}(\ps) \, .
\end{equation}

An important observation is that, in this case, the construction of a family of subspaces of finite measure (cf. discussion after Assumption \ref{as:approximation-property-of-Xi}) is not needed. Indeed, we notice that
$$
\cI_\veps[\sJ_{+,\delta}(z,\gamma;\om_\veps(\ps))\rho]\leq\cI[\sJ_{+,\delta}(z,\gamma;\om_\veps(\ps))\rho]=\int_\Xi(g(s(\xi+\eta_\veps(\xi)+\ps(\xi)))-\gamma)_{+,\delta}\,d\rho\, .
$$
This, according to Assumption \ref{as:growth-condition-on-gos} and the discussion following it, shows that the integral is finite. Furthermore, in a situation where Assumption \ref{as:on the embedding} is satisfied in $\Xi$ (for instance, when $\Xi$ is finite-dimensional), there is no need whatsoever to consider the sequence $(\Xi_\veps)_\veps$. Namely, Assumption \ref{as:approximation-property-of-Xi} is not needed.

In the remainder of this section, we state the main theorems given in Section \ref{sec:unified analysis}, with the corresponding simplifications, which are certainly corollaries of the unified analysis.

\begin{theorem}
Fix $\veps>0$ and $\delta>0$. Then, there exists
$$
(z^\star,\gamma^\star,\ps^\star)\in \rZ_\ad\times \R\times \bT_\veps
$$
such that
$$
\Phi_\veps^\delta(z^\star,\gamma^\star;\ps^\star)\leq \Phi_\veps^\delta(z,\gamma;\ps)\qquad \forall (z,\gamma;\ps)\in  \rZ_\ad\times \R\times \rT\, .
$$
\end{theorem}
\begin{proof}
It follows from the proof of Theorem \ref{thm:existence-of-minimizers}, with the aforementioned simplifications.
\end{proof}

Next, we state the first-order optimality conditions for this case. To this end, and under the same notation as the previous sections, instead of item (1) in Assumption~\ref{as:on-differentiability}, we assume that $\cH_\bT$ is Fréchet differentiable, with derivative $\bD\cH_\bT(\ps)\in \sL(\bT,\R)$. The remaining assumptions are unchanged.

The following result constitutes the simplified version of Theorem \ref{thm:optimality-conditions}.
\begin{theorem}
Let $(z_\veps^\star,\gamma_\veps^\star,\ps_\veps^\star)\in \rZ_\ad\times \R\times \bT_\veps$ be a minimizer of the Rockafellian \eqref{eq:Rockafellian with corruption to support}. Setting $\lambda_\veps^\star:=(\omega_\veps(\ps_\veps^\star),z_\veps^\star)$, we have
\begin{enumerate}
\item For all $z\in \rZ_\ad$,
\begin{equation*}
\Big\langle \bD f_0(z_\veps^\star) + \kappa\, \cI_\veps \big[ \wp_\veps(t_\veps^\star)\, A_\delta\big(\sJ(z_\veps^\star,\gamma_\veps^\star;\omega_\veps(\ps_\veps^\star))\big)\, \bD_z s(\lambda_\veps^\star)^*\, \bD g(s(\lambda_\veps^\star)) \big] , z-z_\veps^\star \Big\rangle \geq 0\, .
\end{equation*}
\item It holds
\begin{equation*}
\cI_\veps\Big[ \rho\, A_\delta\big( \sJ(z_\veps^\star,\gamma_\veps^\star;\omega_\veps(\ps_\veps^\star))\big) \Big]  = \frac{1}{\kappa}\, .
\end{equation*}
\item  For all $\ps \in \bT_\veps$, 
\begin{equation*}
\Big\langle \theta_\veps\, \bD_{\ps}\cH_\bT(\ps_\veps^\star) + \kappa \, \cI_\veps\Big[ \rho\, A_\delta\big( \sJ(z_\veps^\star,\gamma_\veps^\star;\omega_\veps(\ps_\veps^\star)) \big)\, \bD_{\ps} s(\lambda_\veps^\star)^*\, \bD g (s(\lambda_\veps^\star) \Big], \ps-\ps_\veps^\star\Big\rangle  \geq 0 \, .
\end{equation*}
\end{enumerate}
\end{theorem}
\begin{proof}
The proof is analogous to that of Lemma~\ref{lem:derivatives-of-Rockafellian} and Theorem~\ref{thm:optimality-conditions}, and therefore is omitted.
\end{proof}

We finish this section by stating the simplified version of Theorem \ref{thm:weak-strong-gamma-convergence-result}, which concerns weak-strong Gamma convergence.
\begin{theorem}
Let $\rho\in \rP$ and consider corruption maps $(\eta_\veps)_\veps$ such that $\eta_\veps\in \bQ_\veps\cap \bT_\veps$, for each $\bT_\veps>0$. Assume that
$$
\lim_{\veps\downarrow 0}\theta_\veps
=+\infty\qan \lim_{\veps\downarrow 0}\, \theta_\veps\, \cH_\bT(I-\eta_\veps)=0 \, .
$$
Then, under the previously discussed assumptions, we have weak-strong $\Gamma$-convergence of $\Phi_\veps^\delta$ to $\Phi^\delta$.
\end{theorem}
\begin{proof}
The proof is essentially identical to that of Theorem \ref{thm:weak-strong-gamma-convergence-result}, with the difference that now, as a result of the absence of corruption to the distribution,  the reasoning leading to \eqref{eq:lim inf bound of J} is greatly simplified. The rest of the proof works verbatim.
\end{proof}

\section{Numerical examples}\label{sec:numerics}
We consider two stochastic optimization problems constrained by elliptic boundary value problems. 
 In each case, we model a different corruption scenario. In the first, samples from a normally distributed random vector 
are corrupted by introducing samples with increased variance. In the second, a truncated 
exponential random variable is corrupted by increasing the likelihood of tail events. 

\subsection{Example 1}\label{subsec:example_1} 
For our first numerical example, we consider the stochastic optimal control problem 
\begin{equation}  \label{eq:1d_optimal_control_problem}
\min_{z \in L^2(0,1)} \varphi(z), \qquad 
\varphi(z) =  \on{CVaR}_{\beta}\left( \frac12 \| s(\xi, z) - u_{\star}\|_{L^2(0,1)}^2 \right) + \frac{\alpha}{2} \| z\|_{L^2(0,1)}^2
\end{equation} 
where $u(\cdot, \xi) := s(\xi, z)$ is the solution operator to the one-dimensional boundary value problem 
\begin{align} 
-\frac{d}{dx} \Big( a(x,\xi) \frac{d}{dx} u(x,\xi) \Big) = z(x)&, \qquad (x,\xi) \in (0,1)\times \Xi \label{eq:1d_bvp_constraint} \\
u(0,\xi) = u(1,\xi) = 0&,  \nonumber
\end{align}
The diffusion coefficient is given by 
$$
a(x,\xi) = e^{\sigma \sum_{k=1}^d \sqrt{\lambda_k} \sin(x/\sqrt{\lambda_k}) \xi_k}, \qquad \lambda_k = \frac{4}{(2k-1)^2 \pi^2}, 
$$
where each $\xi_k$ is independent and normally distributed with mean-zero and unit variance. Here the argument of the exponential 
function is a truncated Kosambi-Karhunen-Lo\'{e}ve expansion \cite[Theorem 2.3]{martinez2018optimal}
for a rescaled Brownian motion on the interval $(0,1)$.  
We take $\sigma =0.4$ and consider $d=50$ expansion terms.  
For the control regularization term in the objective function, we take $\alpha = 10^{-5}$, while 
the target function $u_{\star}(x) = 1$. 

To numerically solve the optimization problem, we discretize the expectation value in the objective
function with a sample-average approximation using $N=5000$ samples. 
Similar to the numerical experiments reported in 
\cite[Section 5.2]{beiser2023adaptive}, we employ 
the ``nested quantile estimation'' (NQE)
strategy described in \cite[Algorithm 2]{beiser2023adaptive}, which 
amounts to an alternating coordinate descent algorithm \cite[Section 1.J]{royset2021optimization} on the 
variables $z$ and $\gamma$.
%
After mollifying the function $(x)_{+}$ with 
\begin{equation}\label{eq:plus_fn_mollification}
(x)^{\delta}_+ := x + \delta \ln(1 + e^{-x/\delta})
\end{equation}
(from \cite[Eqn.~(4.6)]{KS}), the optimization problem for fixed $\gamma$ becomes
$$
\min_{z\in L^2(0,1)} \left(  \frac{1}{1-\beta} \frac{1}{N} \sum_{i=1}^N \Big(\frac12\| s(\xi^{(i)}, z) - u_{\star}\|_{L^2(0,1)}^2 - \gamma \Big)_+^{\delta} 
+ \frac{\alpha}{2} \| z\|_{L^2(0,1)}^2 \right), 
$$
 which is numerically solved using SciPy's implementation of the L-BFGS-B algorithm with the default settings \cite{2020SciPy-NMeth}. 
In turn, for fixed $z$, the optimization problem 
$$
\min_{\gamma \in \R} \left(\gamma +  \frac{1}{1-\beta}\frac{1}{N}\sum_{i=1}^N \Big(\frac12\| s(\xi^{(i)}, z) - u_{\star}\|_{L^2(0,1)}^2 - \gamma \Big)_+^{\delta} 
\right)
$$
is solved using the bisection method to find the critical point. Note that this latter optimization problem 
is relatively cheap; the bulk of the computational expense requires less than a single objective function calculation. 
The NQE algorithm terminates when the absolute difference in optimal $\gamma$ values between successive iterations is smaller 
than a specified tolerance, which we set to $10^{-4}$. In all numerical examples, the mollification parameter
$\delta = 10^{-3}$. 

The state equation \eqref{eq:1d_bvp_constraint} and the corresponding adjoint equation are discretized with 
a second-order finite difference method with uniform grid spacing $h = 1/128$; the resulting tridiagonal 
systems are solved with SciPy's banded direct solver. 
All $L^2$ norms are approximated with the composite trapezoidal rule. 

To introduce data corruption into the stochastic optimization problem 
\eqref{eq:1d_optimal_control_problem}, we select the first $M \in \N$ samples $(M < N)$ 
and increase their variance through the map $\xi \mapsto 5\xi$, which amounts to setting 
$\sigma = 2$ for those samples.  

For the values $\beta\in \{0.1, 0.5, 0.9\}$, we 
solve both the uncorrupted version of \eqref{eq:1d_optimal_control_problem}
and the corrupted version at various corruption levels,  
defined as 
\begin{equation} \label{eq:corrupted_amount_defn_case2}
 \text{\% corruption} := 100 \cdot M/N . 
\end{equation}
Figure \ref{fig:control_plots_example_1} shows that even with only 1\% of the samples corrupted, 
the optimal control $z^{\ast}_{\rm corrupted}$ is significantly different 
than the uncorrupted one $z^{\ast}_{\rm true}$ for all values of $\beta$. The differences 
grow with increasing corruption.

\begin{figure}[ht]       
    \centering
 \includegraphics[width=0.32\textwidth]{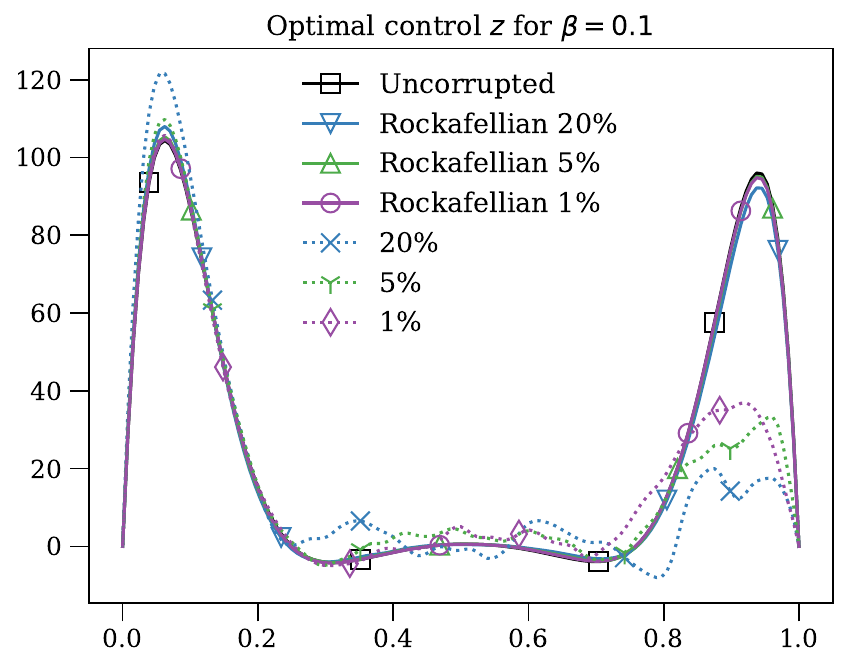}
 \includegraphics[width=0.32\textwidth]{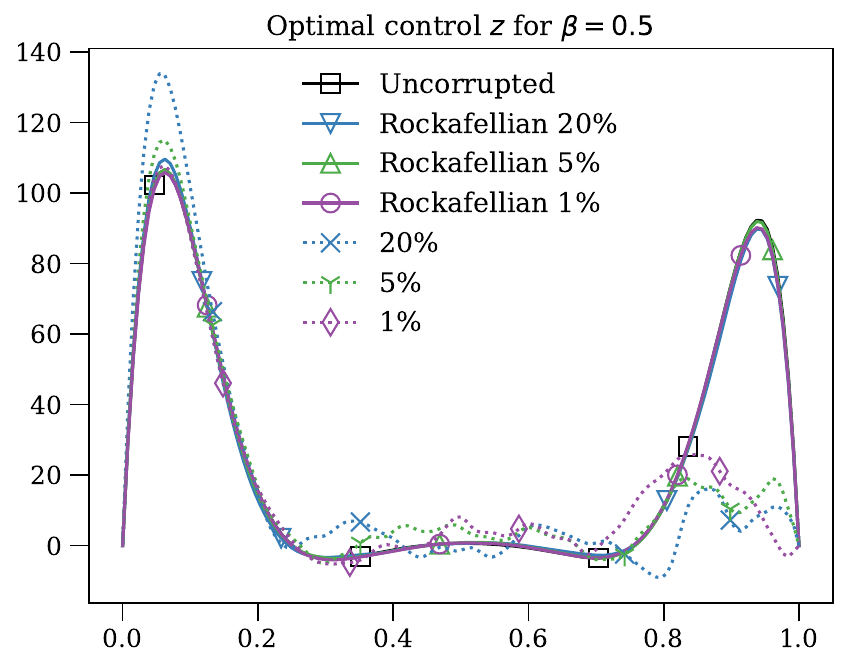}
 \includegraphics[width=0.32\textwidth]{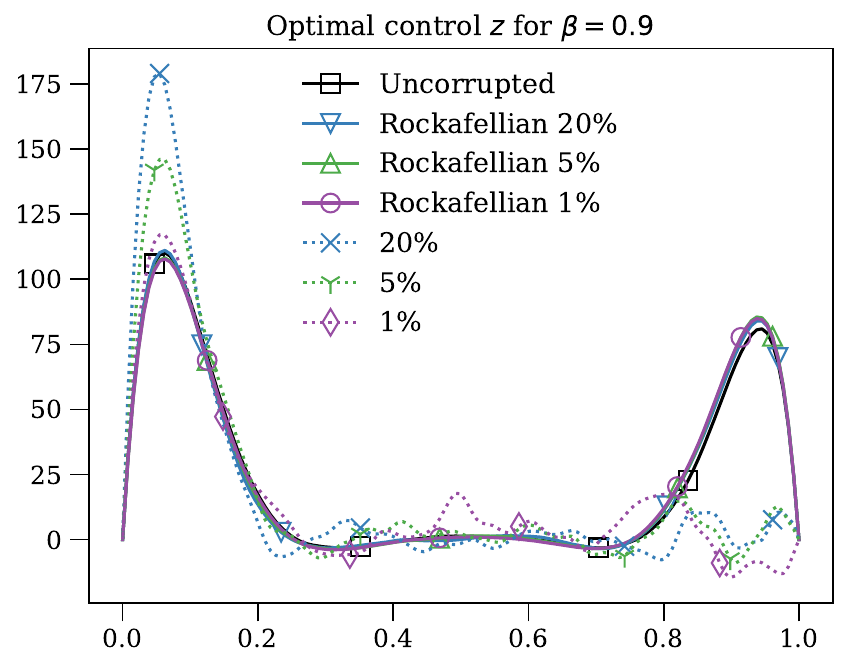}
\caption{ 
Example 1: Optimal controls for the true, uncorrupted problem, as well as for the corrupted
problem (dotted lines) and the corresponding Rockafellian relaxations at varying 
corruption levels.  
} 
\label{fig:control_plots_example_1}
\end{figure}

To recover $z^{\ast}_{\rm true}$ in the presence of data corruption, we define the Rockafellian
\begin{equation} \label{eq:Rock_example_1} 
\Phi(z,t) := \min_{\gamma \in \R}  \left( \gamma + \frac{1}{1-\beta}\sum_{i=1}^N \Big( \frac{1}{N} + t_i\Big) 
\Big(\frac12 \| s(\xi^{(i)}, z) - u_{\star}\|^2_{L^2(0,1)} - \gamma\Big)^{\delta}_+  \right) + \theta \| t\|_1 + \iota_{\Delta} (p+t). 
\end{equation}
where $t \in \R^{N}$, $\theta > 0$, $\Delta$ is the set of probability vectors, and every component of $p \in \R^N$ equals $1/N$. 
Following the numerical experiments in \cite{royset2025rockafellian} and \cite{ACDR}, 
the $l_1$ norm is chosen for the perturbation variable $t$ because of its well known sparsity-promoting property.
We also borrow from \cite{royset2025rockafellian} and \cite{ACDR} the numerical optimization strategy 
to minimize $\Phi$, which involves an alternating-direction (ADI) heuristic in which we first fix $t=0$ and compute
$z^{\ast} \in \argmin \Phi(z,0)$ using the NQE described above. Using the result, we then solve the linear program 
$t^{\ast} \in \argmin \Phi(z^{\ast}, t)$ using SciPy's implementation of the simplex method. This process is repeated until the 
absolute $l_1$ distance between successive $t^{\ast}$ values is smaller than $10^{-3}$. 

The constraint that $p + t \in \Delta$ implies the 
pointwise bounds $ -1/N \le t_i \le 1-1/N$ for each $1\le i \le N$. As in the
numerical examples in \cite{ACDR}, in practice we find better results after imposing the tighter bounds $ -1/N \le t_i \le 1/N$, 
as this prevents the linear program solver from ``greedily'' identifying a small handful of samples $\xi^{(i)}$ at which 
$$
\Big( \frac12 \| s(\xi^{(i)}-u_{\star}\|_{L^2(0,1)}^2 \Big)^{\delta}_+ 
$$
is smallest and then subsequently removing the other samples (including potentially uncorrupted, ``clean'' samples). 

Figure \ref{fig:control_plots_example_1} shows that the optimal controls $z^{\ast}_{\rm Rock}$ obtained by minimizing 
\eqref{eq:Rock_example_1} with $\theta = 10^{-1}$ are quite close to $z^{\ast}_{\rm true}$ for all 
values $\beta \in \{0.1, 0.5, 0.9\}$ and 1\%, 5\% and 20\% corruption levels. 
The differences are quantified in Table \ref{table:example_1_accuracy}, which shows the 
relative $L^2$ errors and ratio of errors, as defined by 
\begin{equation} \label{eq:error_defns}
E_{\rm rel}(z) := \frac{\| z - z_{\rm true}^{\ast}\|_{L^2(0,1)}^2 }{\| z_{\rm true}^{\ast}\|_{L^2(0,1)}^2} 
\quad \text{and} \quad 
\cE_{\rm ratio}:=  \frac{ E_{\rm rel}(z^{\ast}_{\rm corrupted}) }{E_{\rm rel}(z^{\ast}_{\rm Rock})}, 
\end{equation}
respectively. 
The $L^2$ errors for $z^{\ast}_{\rm Rock}$ are all at least one order of magnitude smaller than the corresponding 
ones in the corrupted cases. Rockafellian relaxation can even mitigate data corruption at
  a 40\% level (not illustrated in Figure \ref{fig:control_plots_example_1}), 
although the reduction in error is not as substantial, particularly for $\beta = 0.9$. 
Additionally displayed are the percentages of corrupted and clean data samples removed by 
the Rockafellian perturbation variable $t$. For fixed corruption level, the results indicate
that, as $\beta$ increases, 
the Rockafellian variable $t$ removes more corrupted data samples.  
In general, both the $L^2$ errors and the number of removed data samples will change as a function 
of $\theta$.

\begin{table}                                   
  \begin{center}
  \def~{\hphantom{0}}
    \begin{tabular}{ c  c  c  c  c  }
\hline
Corruption level & $E_{\rm rel}(z^{\ast}_{\rm Rock})$ & $\mathcal{E}_{\rm ratio}$ & Corrupted deleted & Clean deleted \\
\hline\hline
 & & $\beta=0.1$ & & \\                       
\hline\hline
  1\%            & $7.50\cdot 10^{-2}$ & 48.0  &  14/50=28\% &  1/4950=0.02\%                     \\
  5\%            & $9.38\cdot 10^{-2}$ & 40.4   &  85/250=34\% &   0/4750=0.0\%                     \\  
  20\%            & $3.55\cdot 10^{-1}$ & 13.7 &  352/1000=35.2\% & 0/4000=0.0\%                     \\
  40\%           & $7.13\cdot 10^{-1}$ & 7.76 &  645/2000=32.3\% & 0/3000=0.0\%                     \\
\hline\hline
& & $\beta=0.5$ & & \\                       
\hline\hline
  1\%            & $1.36\cdot 10^{-1}$ & 35.4  &  21/50=42\% &  11/4950=0.22\%                     \\
  5\%            & $6.06\cdot 10^{-2}$ & 76.3 &  123/250=49.2\% &   11/4750=0.23\%                     \\  
  20\%            & $2.83\cdot 10^{-1}$ & 18.6 &  508/1000=50.8\% & 9/4000=0.25\%                     \\
  40\%           & $4.22\cdot 10^{-1}$ & 13.7 &  972/2000=48.6\% & 5/3000=0.17\%                     \\
\hline\hline
& & $\beta=0.9$ & & \\                       
\hline\hline
  1\%            & $4.92\cdot 10^{-1}$ & 11.7  &  33/50=66\% &  199/4950=4.02\%                     \\
  5\%            & $4.79\cdot 10^{-1}$ & 10.8 &  177/250=70.8\% &   194/4750=4.08\%                     \\  
  20\%            & $4.51\cdot 10^{-1}$ & 13.2 &  693/1000=69.3\% & 139/4000=3.48\%                     \\
  40\%           & $1.93\cdot 10^{0}$ & 3.31 &  1103/2000=55.2\% & 51/3000=1.70\%                     \\
    \end{tabular}
    \caption{
Example 1 with $\theta = 10^{-1}$: 
relative $L^2$ errors $E_{\rm rel}$  
between the Rockafellian and true optimal controls, as well as 
the ratio of $L^2$ errors $\mathcal{E}_{\rm ratio}$ for the corrupted and 
Rockafellian optimal controls; see \eqref{eq:error_defns}. Also included
are the fraction of corrupted and 
clean sample points correctly and incorrectly removed by the perturbation 
variable $t$, respectively. 
Corruption levels defined by \eqref{eq:corrupted_amount_defn_case2}. 
}
  \label{table:example_1_accuracy}
  \end{center}
\end{table}

Table \ref{table:example_1_theta} illustrates the impact of varying $\theta$ when $\beta=0.9$ and 
the corruption level is fixed at 10\%. 
As $\theta$ decreases, the number of corrupted data samples that are removed increases; however, 
the number of uncorrupted, clean samples removed also increases. For this particular problem, 
the value of $\theta=10^{-1}$ yields the largest reduction in $L^2$ error. For all values of $\theta$, 
however, the Rockafellian optimal control is closer to the true, uncorrupted one than the corresponding
corrupted controls. 

\begin{table}                                   
  \begin{center}
  \def~{\hphantom{0}}
    \begin{tabular}{ c | c | c | c | c  }
\hline
$\theta$ & $E_{\rm rel}(z^{\ast}_{\rm Rock})$ & $\mathcal{E}_{\rm ratio}$ & Corrupted deleted & Clean deleted \\
\hline
  1                   & $1.75\cdot 10^{0}$ & 3.12  &  93/500=18.6\% &  0/4500=0.0\%                     \\
  $10^{-1}$           & $4.93\cdot 10^{-1}$ & 11.1   &  359/500=71.8\% &   179/4500=3.98\%                     \\  
  $10^{-2}$           & $1.18\cdot 10^{0}$ & 4.62 &  451/500=90.2\% & 1853/4500=41.2\%                     \\
  $10^{-3}$           & $9.70\cdot 10^{-1}$ & 5.64 &  445/500=89.0\% & 2038/4500=45.3\%                     \\
    \end{tabular}
    \caption{
Results from Example 1 at various $\theta$ values with $\beta =0.9$ and 10\% corruption 
(as defined by \eqref{eq:corrupted_amount_defn_case2}). $E_{\rm rel}$ and $\cE_{\rm ratio}$ are 
defined by \eqref{eq:error_defns}. 
}
  \label{table:example_1_theta}
  \end{center}
\end{table}

\subsection{Example 2} 
For our second numerical example, we consider the stochastic optimal control problem 
\begin{equation}  \label{eq:2d_optimal_control_problem}
\min_{z \in \Omega} \varphi(z), \qquad 
\varphi(z) =  \on{CVaR}_{\beta}\left( \frac12 \| s(\xi, z) - u_{\star}\|_{L^2(\Omega)}^2 \right) + \frac{\alpha}{2} \| z\|_{L^2(\Omega)}^2
\end{equation} 
where $u(\cdot, \xi) := s(\xi, z)$ is the solution operator to the two-dimensional boundary value problem 
\begin{align} 
-\Delta u(x,\xi) + v(\xi) \cdot \nabla u(x,\xi) = f(x) + z(x)&, \qquad (x,\xi) \in \Omega \times \Xi \label{eq:2d_bvp_constraint}\\ 
u(x,\xi) =  0&, \qquad (x,\xi) \in \partial \Omega \times \Xi,  \nonumber
\end{align}
where $\Omega = B_1(0)$, i.e.\ the unit-ball in $\mathbb{R}^2$ centered at the origin. 
The divergence-free advection field is given by 
$$
v(\xi) = \Big\langle \xi \cos\big(\xi \pi/v_{\rm max} -\pi/2 \big) ,  
\xi \sin\big(\xi \pi /v_{\rm max} - \pi/2\big)  \Big\rangle
$$
where $v_{\rm max}$ equals 20. The target function in the objective function $u_{\star}(x) =1$. 

The random variable $\xi$ takes values on the state space $\Xi = [0,20]$.
In the true, uncorrupted case, $\xi$ is distributed according to a truncated exponential probability 
density function given by 
\begin{equation} \label{eq:rho_true}
\rho_{\rm true}(\xi) = \Big(\frac{k }{1-e^{-k v_{\rm max}}} \Big)e^{-k \xi}
\end{equation}
where $k = 0.25$. 
Intuitively, at smaller values of $\xi$, the advection field is weaker and points more towards the south, while at 
larger values of $\xi$, the advection field is stronger and points more towards the north. 
We refer to this case as 0\% corruption. 

In the fully (100\%) corrupted case, we consider $\xi$ with a density function that decays algebraically, 
rather than exponentially: 
$$
\rho_{\rm corrupted}(\xi) = \Big( \frac{1}{\log(v_{\rm max}/a + 1)}  \Big) \frac{1}{\xi + a}, 
$$
where $a = 5$.  
Finally, we also consider a partially corrupted case, in which the probability density of $\xi$ 
is given by $\frac12 (\rho_{\rm true} + \rho_{\rm corrupted})$, which we refer to as 50\% corruption. 
The three cases are shown in Figure \ref{fig:density_plots}; notice that the corrupted densities have
 heavier tails than the uncorrupted one. 

\begin{figure}[ht]
 \includegraphics[width=0.5\textwidth]{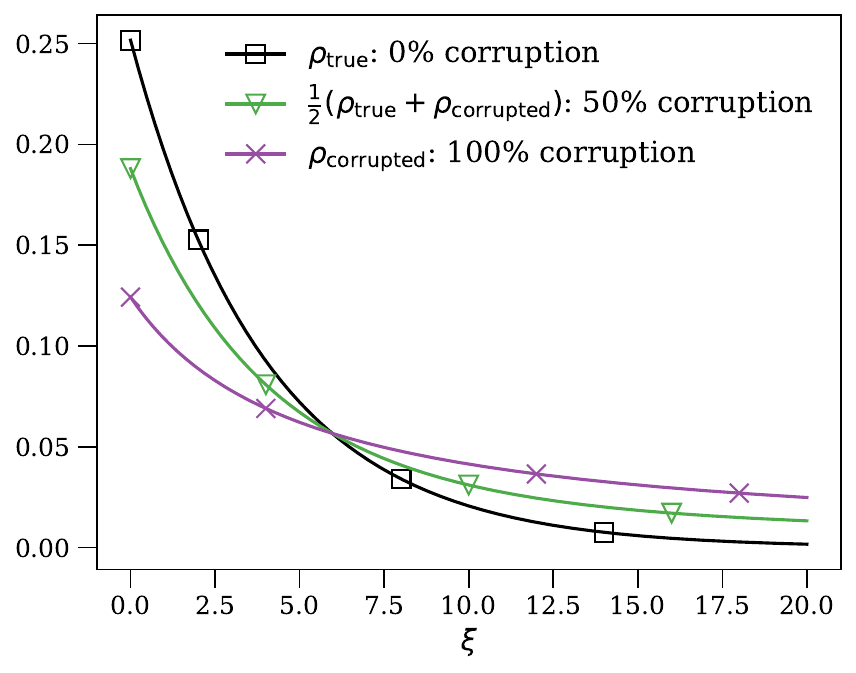}
\label{fig:density_plots}
\caption{Probability density functions for the random variable $\xi$ in the advection field $v(\xi)$ 
for the PDE constraint \eqref{eq:2d_bvp_constraint}. 
}
\end{figure}

To numerically solve the optimization problem \eqref{eq:2d_optimal_control_problem}, 
we discretize the expectation value in the objective function with standard Gaussian 
quadrature (GQ) with $N = 15$ points. The control regularization parameter is set
as $\alpha = 10^{-4}$.  
The same NQE strategy described in the previous subsection is employed to minimize 
over the control variable $z$ and $\on{CVaR}_{\beta}$ variable $\gamma$. The $(x)_+$ is
again mollified with \eqref{eq:plus_fn_mollification} and $\delta = 10^{-3}$. For fixed
$\gamma$, minimization over $z$ is done with the L-BFGS method with history size $m=9$
\cite[Chapter 7.2]{nocedal1999numerical} and backtracking line search based on the Wolfe 
criterion. 

The state equation \eqref{eq:2d_bvp_constraint} and the corresponding adjoint equation
are discretized with a finite element method with $Q1$ elements implemented in 
\texttt{deal.II} 
\cite{arndt2022deal}; the total number of degrees of freedom is 5185.

\begin{figure}[ht]
   \begin{center}
   \includegraphics[width=0.95\textwidth]{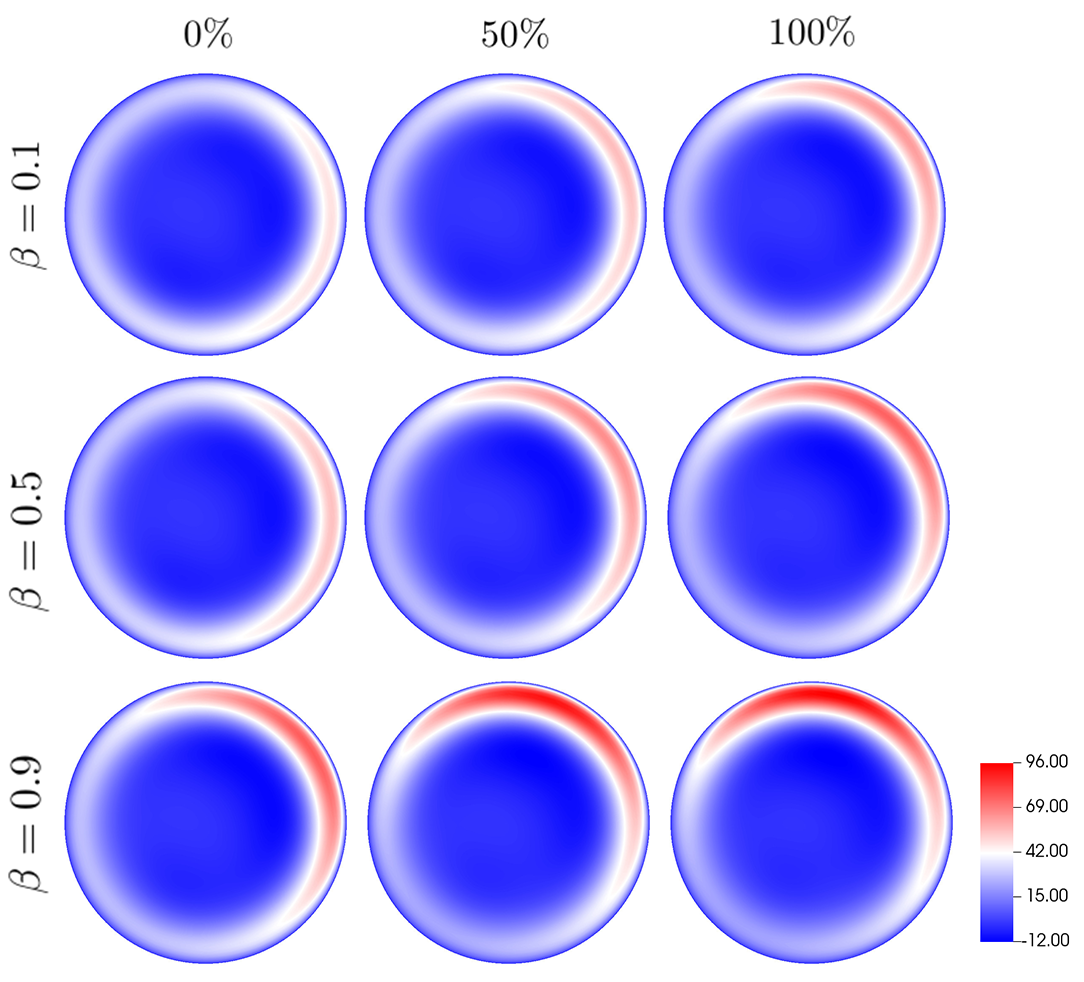}
   \end{center}
\caption{
Optimal controls $z^{\ast}$ for \eqref{eq:2d_optimal_control_problem} 
(without any Rockafellian relaxation)
for differing values of risk-tolerance $\beta$ and corruption levels.  
All plots use the same legend. 
}
\label{fig:controls_2d_no_rock}
\end{figure}

Figure \ref{fig:controls_2d_no_rock} shows the computed optimal control $z^{\ast}$ 
for the values $\beta \in \{0.1, 0.5, 0.9\}$ at 0\%, 50\% and 100\% corruption. 
For fixed $\beta$ and increasing corruption level, 
the maximum value of the optimal control increases, and it becomes increasingly localized to the 
north-eastern portion of the domain $\Omega$. 
Increasing $\beta$ at fixed corruption level has a similar effect. 
This is unsurprising, given 
the nature of the corruption (increased probabilities for tail events), and the fact that 
$\on{CVaR}_{\beta}$ averages over worst-case scenarios. 

To mitigate the influence of the heavy tails, particularly at larger $\beta$ values, we define the 
Rockafellian 
\begin{align}  
\Phi(z,t) := \min_{\gamma \in \R}  \Big( \gamma + \frac{1}{1-\beta}&
\int_{\Xi}  \Big(\frac12 \| s(\xi^{(i)}, z) - u_{\star}\|^2_{L^2(0,1)} - \gamma\Big)^{\delta}_+
(\rho(\xi) + t(\xi) ) d\xi 
  \Big)  \nonumber \\
+ &\theta \int_{\Xi} |t(\xi)| d\xi  + \iota_{\textrm{P}} (\rho+t),   \label{eq:Rock_example_2}
\end{align}
where $\theta >0$. The density $\rho$ is given by either $\rho_{\rm true}$, $\rho_{\rm corrupted}$, or 
$\frac12 (\rho_{\rm true} + \rho_{\rm corrupted})$, depending on the corruption level, and
the indicator function $\iota_{\textrm{P}}$ enforces that $\rho+t$ is a probability density 
function. The $L^1(\Xi)$ norm of $t$ is approximated with the same $N=15$ GQ nodes and 
weights as the expectation value.  

The same ADI heuristic described in the prior subsection is used here to minimize the bivariate function 
\eqref{eq:Rock_example_2}. After discretizing the integrals over $\Xi$,
the linear program that results when optimizing over the perturbation variable $t$ 
is solved with the GNU Linear Programming Kit (\texttt{glpk}) \cite{glpk}. As before, we enforce
the pointwise bound $-\rho(\xi_k) \le t(\xi_k) \le \rho(\xi_k)$ at each of the GQ collocation points
$1 \le k \le N$.  

\begin{figure}[ht]
   \begin{center}
       \includegraphics[width=0.30\textwidth]{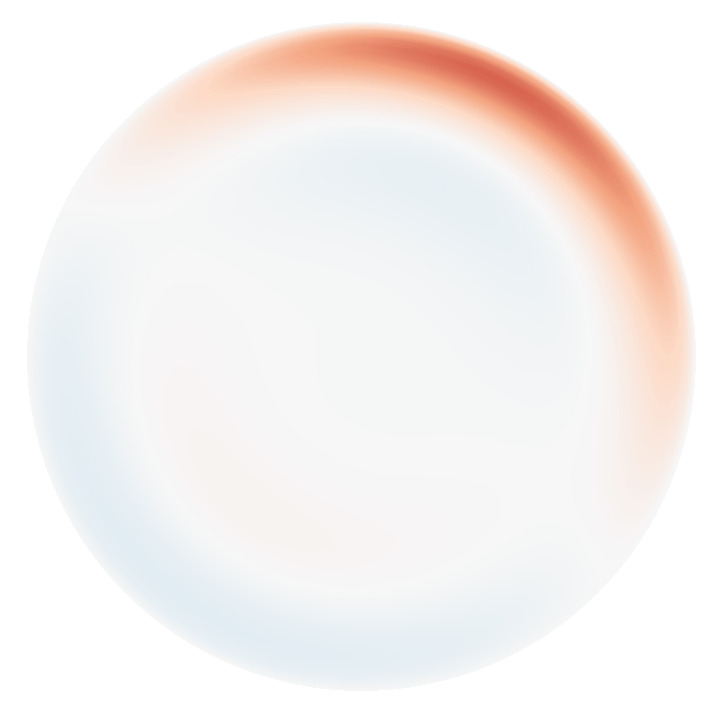}
       \includegraphics[width=0.30\textwidth]{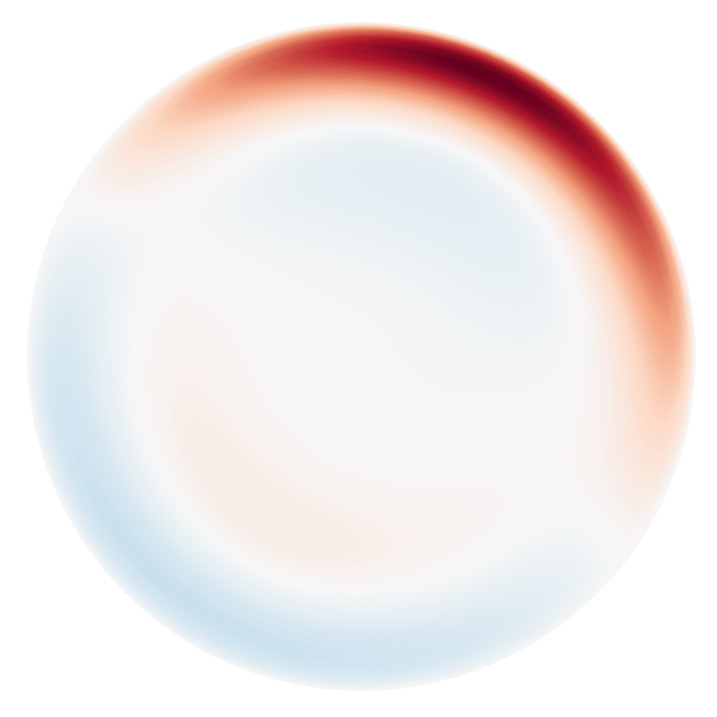}
       \includegraphics[width=0.30\textwidth]{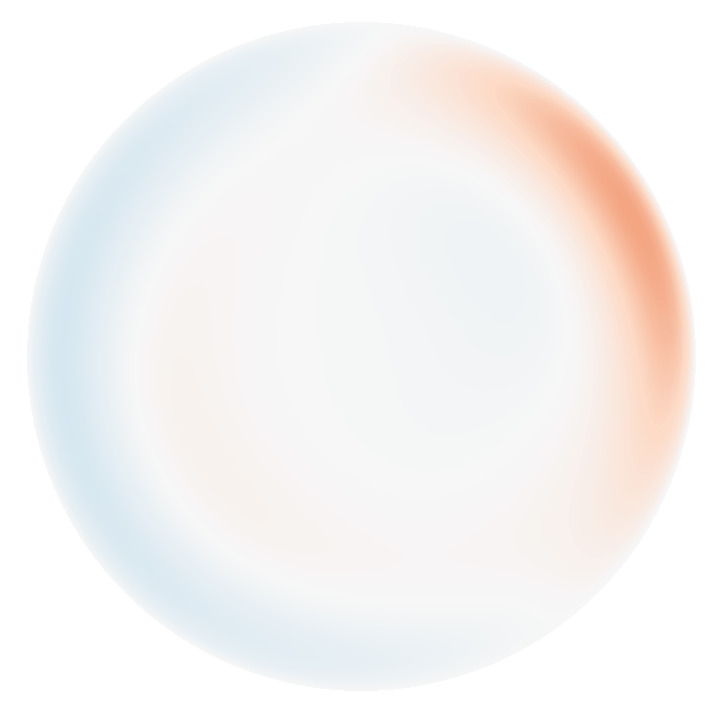}
       \includegraphics[width=0.06\textwidth]{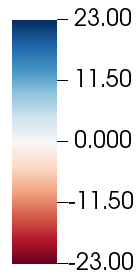}  \\
   \end{center}
\caption{
For $\beta=0.1$: pointwise errors between an uncorrupted optimal control and corrupted optimal controls 
at 50\% and 100\% corruption (left and middle, respectively), as well as the pointwise 
error for a Rockafellian optimal control at 100\% corruption (right). 
All plots use the same legend. 
}
\label{fig:rock_controls_2d_beta0pnt1}
\end{figure}

The pointwise error between (i) the computed minimizer $z^{\ast}$ of the Rockafellian \eqref{eq:Rock_example_2} 
for $\beta=0.1$, 
$\theta = 10^{-1}$, and 100\% corruption and (ii) the uncorrupted optimal control (the computed minimizer of 
\eqref{eq:2d_optimal_control_problem}) is shown in 
Figure \ref{fig:rock_controls_2d_beta0pnt1}, while the analogous error for $\beta =0.9$ is shown in 
Figure \ref{fig:rock_controls_2d_beta0pnt9}. 
Compared with the pointwise errors for the corrupted optimal controls at 50\% and 100\% (also 
shown in Figures \ref{fig:rock_controls_2d_beta0pnt1} and \ref{fig:rock_controls_2d_beta0pnt9}), 
we see that the Rockafellian errors are smaller in magnitude, particularly near the boundary of 
the domain $\partial \Omega$. 

\begin{figure}[ht]
   \begin{center}
       \includegraphics[width=0.30\textwidth]{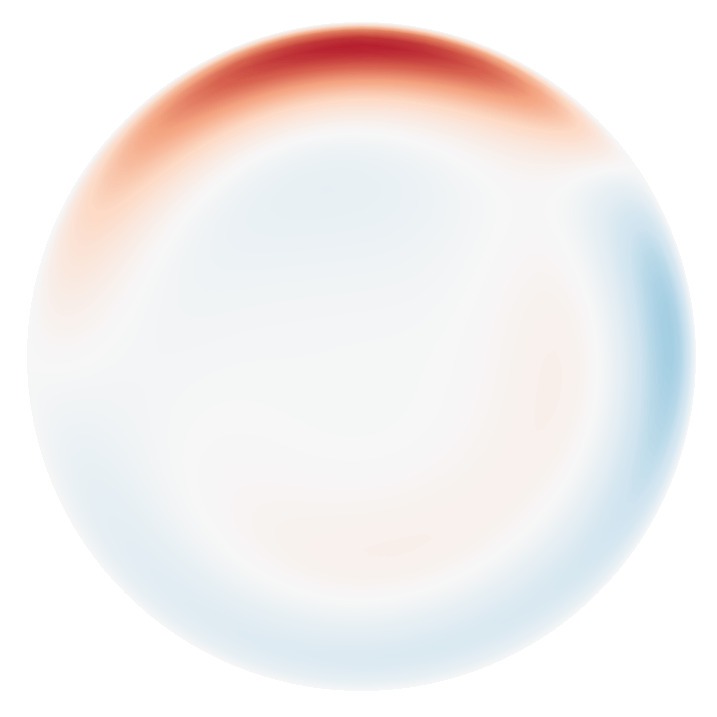}
       \includegraphics[width=0.30\textwidth]{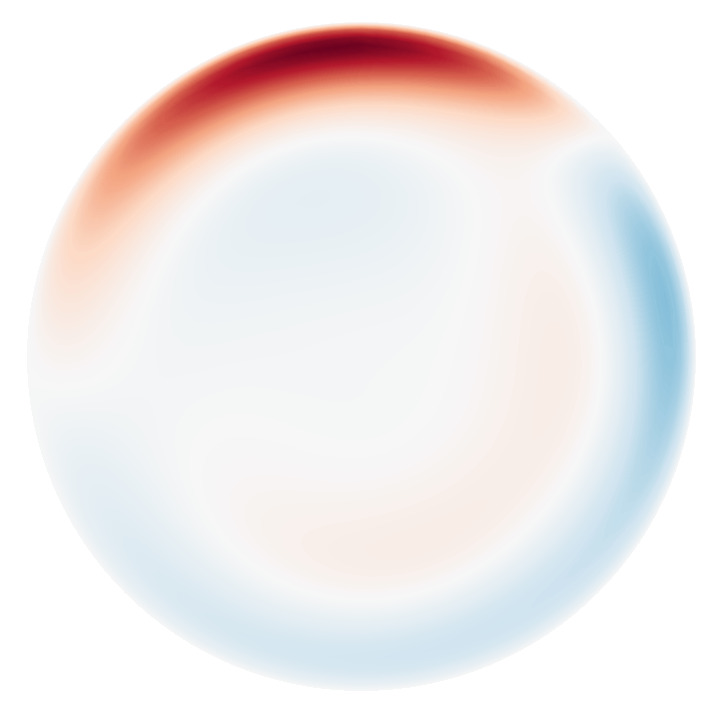}
       \includegraphics[width=0.30\textwidth]{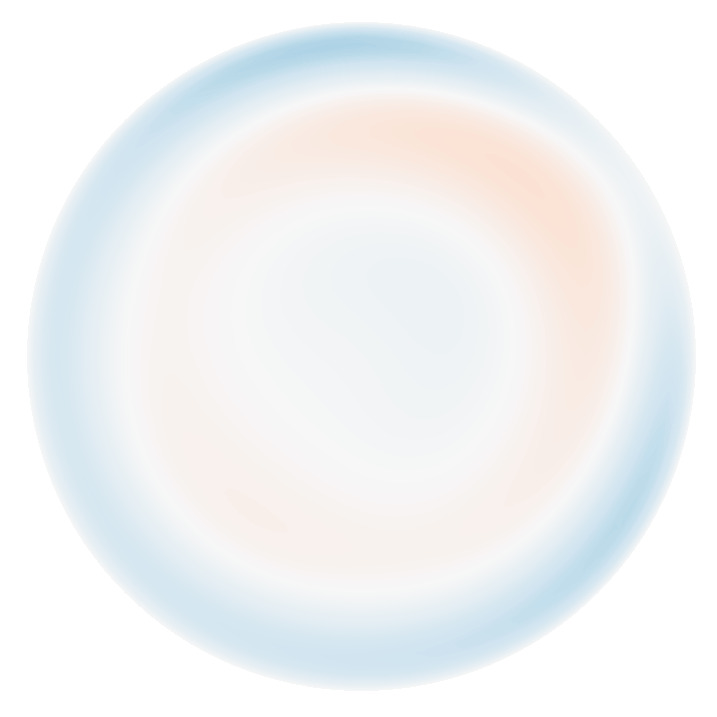}
       \includegraphics[width=0.06\textwidth]{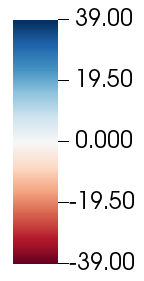}  \\
   \end{center}
\caption{
For $\beta=0.9$: pointwise errors between an uncorrupted optimal control and corrupted optimal controls 
at 50\% and 100\% corruption (left and middle, respectively), as well as the pointwise 
error for a Rockafellian optimal control at 100\% corruption (right).
All plots use the same legend. 
}
\label{fig:rock_controls_2d_beta0pnt9}
\end{figure}

The relative errors in $L^2(\Omega)$ and the ratio of errors--defined analogously to \eqref{eq:error_defns}--are 
quantified in Table \ref{table:example_2_accuracy} for various $\beta$ values and corruption levels. 
Additionally, we show
the values for $\theta \in \{10^{-1}, 10^{-2}\}$. At $\theta =1$, the optimal value of the perturbation variable is
$t^{\ast}(\xi) = 0$, and hence $z^{\ast}_{\rm Rock}$ and $z^{\ast}_{\rm corrupted}$ are the same.
The results for this case are therefore not reported in Table \ref{table:example_2_accuracy} (by definition, 
the error ratio would equal one). 
Employing Rockafellian relaxation meaningfully reduces the errors in the optimal 
controls in the presence of data corruption, although the error ratios $\cE_{\rm rel}$ are not as dramatic 
as those in the previous numerical example (cf.\ Table \ref{table:example_1_accuracy}). 
%
This is unsurprising, given that the corrupted probability densities in this example feature increased likelihoods
of tail events, rather than endogeneous outliers.  

One noteworthy result occurs for $\beta=0.1$, $\theta=10^{-2}$ and 50\% corruption: notice
that the Rockafellian optimal control 
actually is \emph{slightly worse} than the corresponding corrupted optimal control, in the sense that 
the ratio $\cE_{\rm ratio}$ is less
than one. Although the overall relative error in this case is 
 about 16\% for the Rockafellian optimal control, which is in line with the other cases, 
this example illustrates that the ``greediness'' inherent to Rockafellian relaxation 
(and, more generally, DOO) can sometimes lead to suboptimal outcomes, particularly 
when employed in a context where the impact of data corruption is minimal.


\begin{table}                                   
  \begin{center}
  \def~{\hphantom{0}}
    \begin{tabular}{ c  c  c  c  }
\hline
$\beta$ & Corruption level &  $E_{\rm rel}(z^{\ast}_{\rm Rock})$ & $\mathcal{E}_{\rm ratio}$  \\
\hline\hline
 & & $\theta=10^{-1}$ \quad\quad\quad\quad\quad\quad\quad\quad   & \\                       
\hline\hline
0.1 &  50\%           &  $4.74\cdot 10^{-2}$ &   3.05                     \\  
0.9 &  50\%           &  $2.06\cdot 10^{-1}$ &    1.22                    \\  
0.1 &  100\%          &  $1.07\cdot 10^{-1}$ &    2.39                  \\
0.9 &  100\%        &  $1.50\cdot 10^{-1}$   &    2.09                  \\
\hline\hline
 & & $\theta=10^{-2}$ \quad\quad\quad\quad\quad\quad\quad\quad  & \\                       
\hline\hline
0.1 &  50\%           & $1.64\cdot 10^{-1}$  &  0.883                    \\  
0.9 &  50\%           & $2.06\cdot 10^{-1}$  &  1.22                    \\  
0.1 &  100\%        &  $1.34\cdot 10^{-1}$   &  1.92                  \\
0.9 &  100\%        &  $2.12\cdot 10^{-1}$   &  1.49                   \\
    \end{tabular}
    \caption{
For Example 2: relative $L^2$ errors $E_{\rm rel}$ 
between the Rockafellian and true, uncorrupted optimal controls, as well as 
the ratio of $L^2$ errors $\mathcal{E}_{\rm ratio}$ for the corrupted and 
Rockafellian optimal controls--both $E_{\rm rel}$ and $\mathcal{E}_{\rm rel}$ are defined analogously to \eqref{eq:error_defns}. 
The quantities are shown at various combinations of 
corruption level, risk-tolerance $\beta$, and relaxation parameter $\theta$. 
}
  \label{table:example_2_accuracy}
  \end{center}
\end{table}

From this observation, an interesting question naturally arises: what happens when Rockafellian relaxation 
is employed for a risk-averse stochastic optimal control problem with little to no data corruption? 
To answer this and, more generally, to 
further quantify how the combination of Rockafellian relaxation (DOO) and $\on{CVaR}_{\beta}$ (DRO) performs, 
we next analyze the cumulative distribution function (CDF) of the random objective function. 

\subsection{Analysis of cumulative distribution functions} 
First, define the random variable 
\begin{equation}\label{eq:RV_example_2}
j(\xi,z) := \frac12 \| s(\xi, z ) - u_{\star} \|^2_{L^2(\Omega)} 
\end{equation}
associated to the objective function from Example 2, so that 
$s(\xi, z)$ satisfies \eqref{eq:2d_bvp_constraint}, $u_{\star}(x) = 1$, and $\xi$ has 
probability density function given by $\rho_{\rm true}$ (Eqn.~\eqref{eq:rho_true}).

\begin{figure}[ht]
\begin{center}
    \includegraphics[width=0.49\textwidth]{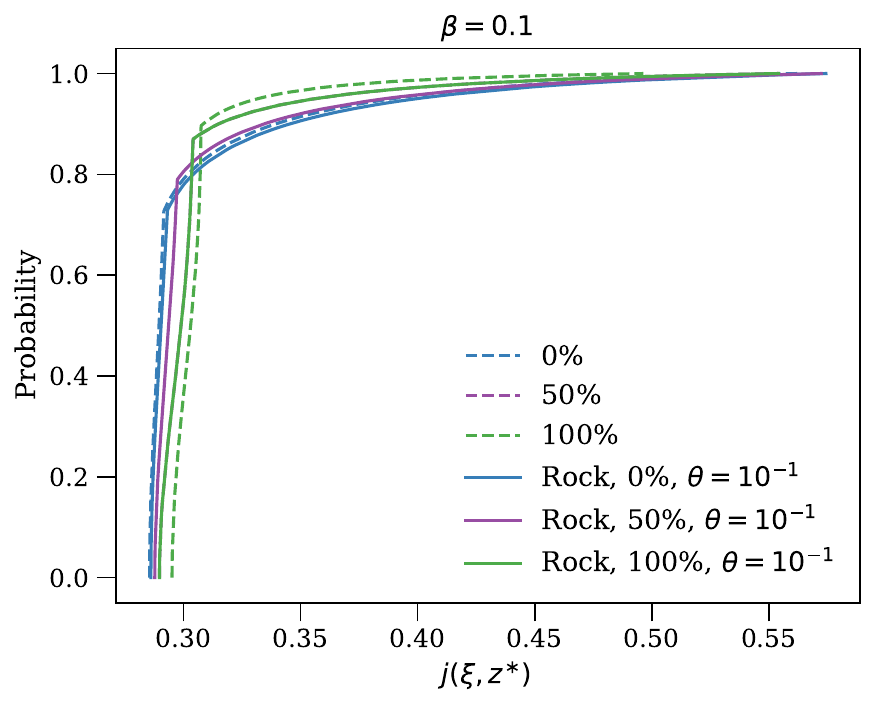}
    \includegraphics[width=0.49\textwidth]{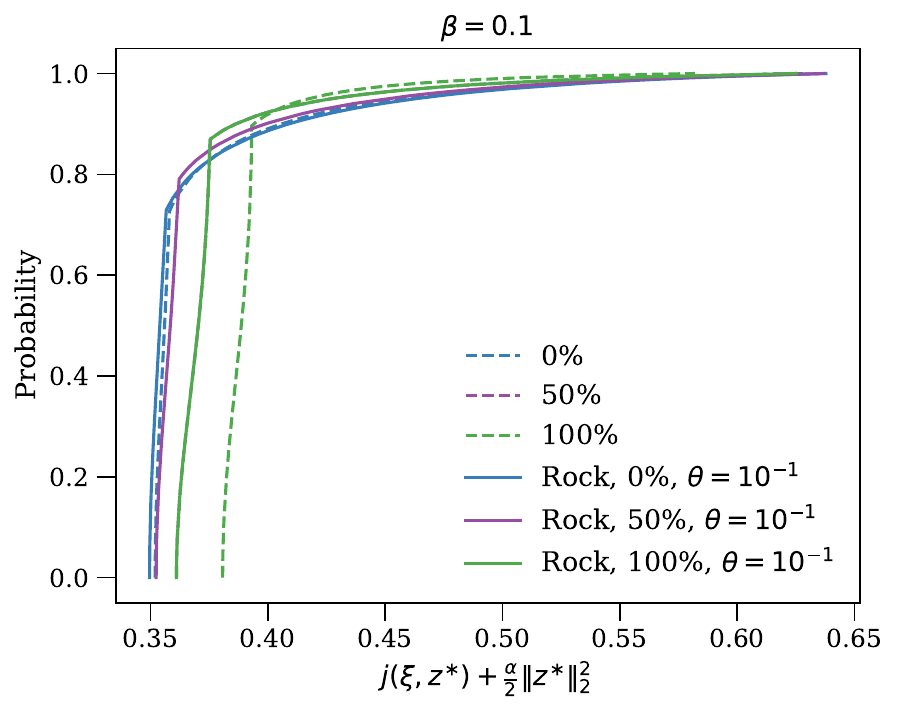}
\end{center}
\caption{
For $\beta=0.1$: CDF plots for the random objective function \eqref{eq:RV_example_2} (left) 
and total control cost (right) for various optimal controls. The dashed lines correspond
to optimal controls without any Rockafellian relaxation.
}
\label{fig:cdf_beta_0pnt1}
\end{figure}

\begin{figure}[ht]
\begin{center}
    \includegraphics[width=0.49\textwidth]{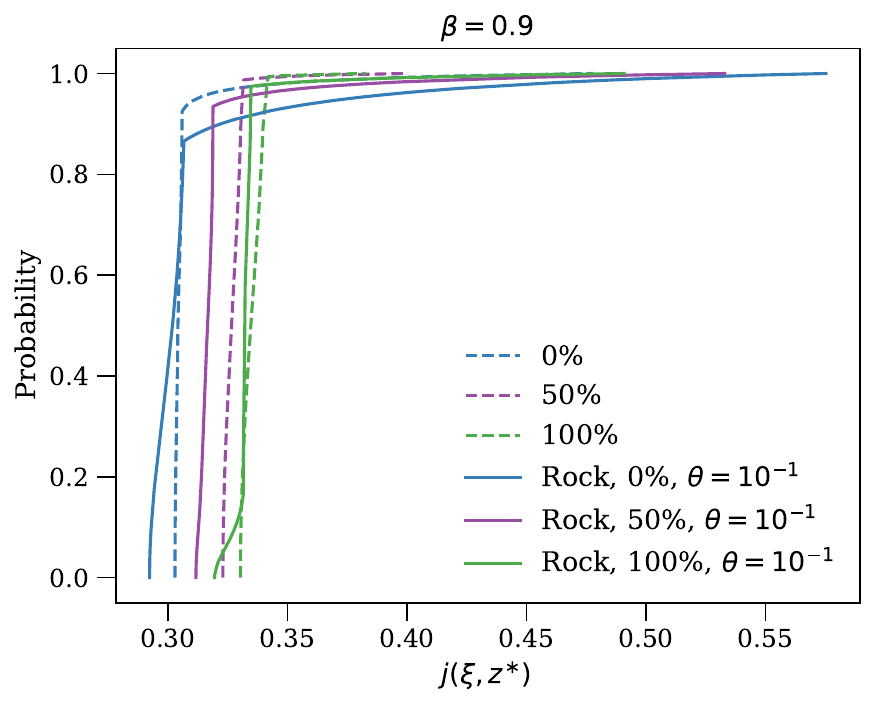}
    \includegraphics[width=0.49\textwidth]{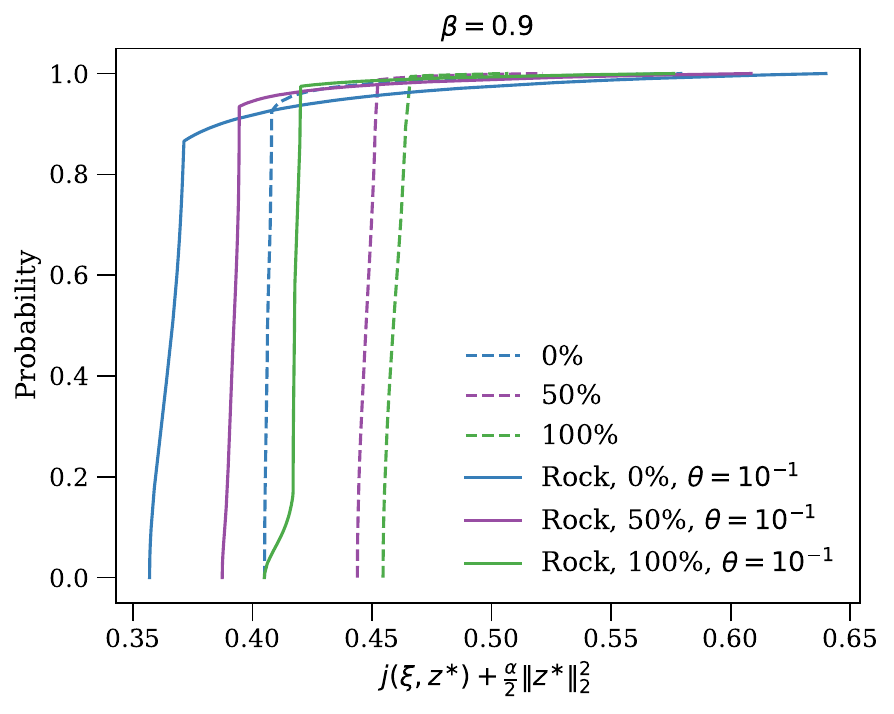}
\end{center}
\caption{
For $\beta=0.9$: CDF plots for the random objective function \eqref{eq:RV_example_2} (left) 
and total control cost (right) for various optimal controls. The dashed lines correspond
to optimal controls without any Rockafellian relaxation.
}
\label{fig:cdf_beta_0pnt9}
\end{figure}

To numerically compute the CDFs at various computed optimal controls, 
\eqref{eq:RV_example_2} is evaluated at $N = 2^{17}$ Sobol 
samples.\footnote{From \texttt{scipy.stats.qmc.Sobol}.} The CDF plots for various
optimal controls at $\beta=0.1$ are shown in Figure \ref{fig:cdf_beta_0pnt1}; 
the CDF of the \emph{total}
control cost ($j(\xi,z) + \frac{\alpha}{2} \| z\|_{L^2(\Omega)}^2$) is also shown for completeness. 
Analogous results for $\beta = 0.9$ are shown in Figure \ref{fig:cdf_beta_0pnt9}. 
In general, for fixed $\beta$ and fixed corruption level, the Rockafellian 
optimal controls are less risk-averse than the non-Rockafellian counterparts:
$j(\xi, z^{\ast})$ is smaller for Rockafellian controls at lower quantiles, but larger
at higher quantiles. 

\begin{figure}[ht]
\begin{center}
    \includegraphics[width=0.49\textwidth]{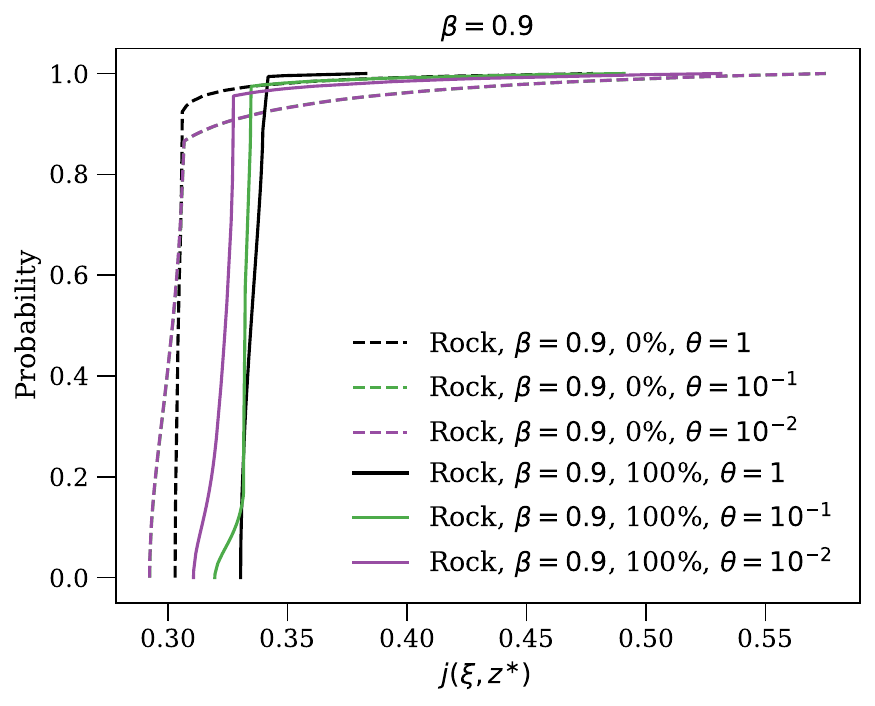}
    \includegraphics[width=0.49\textwidth]{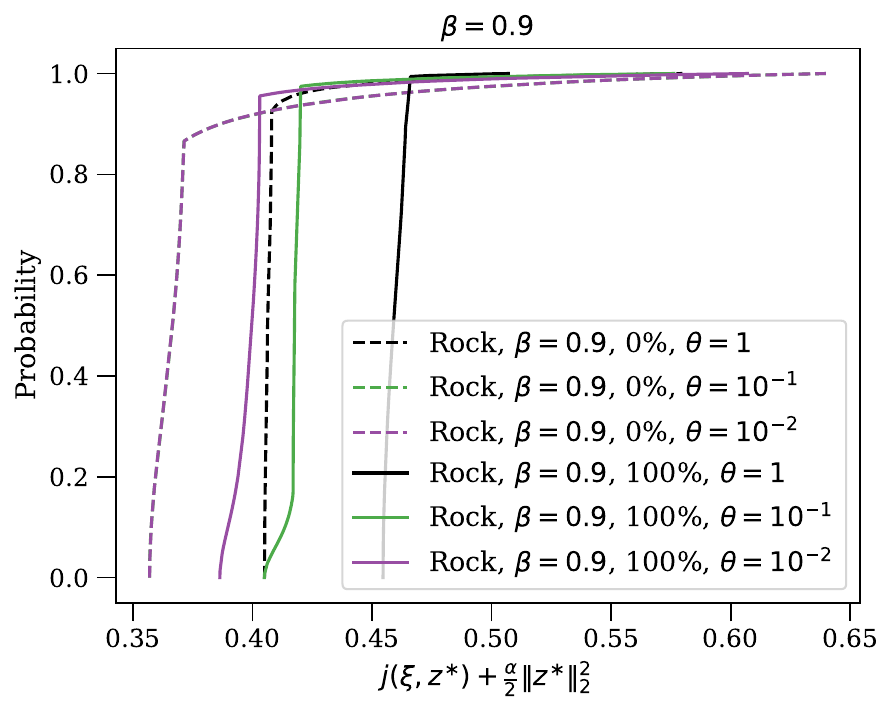}
\end{center}
\caption{
For $\beta=0.9$:  CDF plots for the random objective function \eqref{eq:RV_example_2} (left) 
and total control cost (right) for Rockafellian optimal controls at various $\theta$ values and 
corruption levels. 
}
\label{fig:rock_only_cdf_beta_0pnt9}
\end{figure}

Additionally, the level of risk-aversion of Rockafellian optimal 
controls can be altered by changing $\theta$. Loosely speaking, smaller values of 
$\theta$ correspond to lower confidence in the problem data (in this case, the veracity 
of the given probability density $\rho$); the perturbation variable $t$ is relatively free
to alter the problem data in an optimistic manner, which 
 leads to controls that are less risk-averse. 
On the flip side, larger values of 
$\theta$ correspond to increased confidence in the problem data. In this case, there is less freedom 
for $t$ to modify $\rho$, leading to increased risk-aversion. This is consistent 
with the results of Theorem \ref{thm:weak-strong-gamma-convergence-result}; recall that as the corruption level vanishes (i.e.\ 
as $\rho_{\rm corrupted} \to \rho_{\rm true}$), it is required that $\theta \to \infty$ to 
ensure $\Gamma$-convergence.  Figure \ref{fig:rock_only_cdf_beta_0pnt9} shows the effect of changing $\theta$ on 
the CDFs of the objective function at both 0\% and 100\% corruption. 
As mentioned above, note that at $\theta=1$, the optimal perturbation variable $t^{\ast}(\xi)=0$, 
which means that $z^{\ast}_{\rm Rock} = z^{\ast}_{\rm true}$ and $z^{\ast}_{\rm Rock} = z^{\ast}_{\rm corrupted}$
at 0\% and 100\%, respectively.

\begin{table}                                   
  \begin{center}
  \def~{\hphantom{0}}
    \begin{tabular}{ c c c | c c |  c c c c c  }
\hline
$\beta$ & Corruption level & $\theta$ & Min   & Max   & 0.5   & 0.75  & 0.9   & 0.95  & 0.99  \\
\hline
0.1 & 0\%                  & $10^{-1}$ & 0.3496 & 0.6379 & 0.3540 & 0.3591 & 0.4090 & 0.4630 & 0.5731   \\
0.1 & 0\%                  & --        & 0.3520 & 0.6287 & 0.3560 & 0.3605 & 0.4059 & 0.4556 & 0.5613   \\
0.9 & 0\%                  & $10^{-1}$ & 0.3569 & 0.6398 & 0.3665 & 0.3703 & 0.3866 & 0.4398 & 0.5656   \\
0.9 & 0\%                  & --        & 0.4050 & 0.5798 & 0.4062 & 0.4076 & 0.4079 & 0.4140 & 0.4871   \\
0.1 & 100\%                & $10^{-1}$ & 0.3611 & 0.6254 & 0.3704 & 0.3745 & 0.3863 & 0.4272 & 0.5417   \\
0.1 & 100\%                & --        & 0.3808 & 0.5820 & 0.3887 & 0.3925 & 0.3941 & 0.4173 & 0.4993   \\
0.9 & 100\%                & $10^{-2}$ & 0.3863 & 0.6070 & 0.4000 & 0.4024 & 0.4029 & 0.4029 & 0.5046   \\
0.9 & 100\%                & $10^{-1}$ & 0.4049 & 0.5762 & 0.4175 & 0.4191 & 0.4199 & 0.4200 & 0.4714   \\
0.9 & 100\%                & --        & 0.4546 & 0.5070 & 0.4589 & 0.4626 & 0.4641 & 0.4652 & 0.4659   
    \end{tabular}
    \caption{
For Example 2: Statistics associated with the total control cost $j(\xi,z^{\ast})+\frac{\alpha}{2} \| z^{\ast}\|^2_{L^2(\Omega)}$ at various
corruption levels and $\beta$ values. Wherever a $\theta$ value is indicated, this corresponds to a Rockafellian case.  
Min and Max refer to the minimum and maximum value of the total cost, respectively, while the final 
five columns list various quantiles of the random objective function.
}
  \label{table:example_2_cdf}
  \end{center}
\end{table}

The quantitative performance of the various optimal controls is summarized in Table \ref{table:example_2_cdf}. 
Returning to the question how utilizing $\on{CVaR}_{\beta}$ with Rockafellian relaxation affects stochastic optimal control problems when 
there is little to no data corruption present, we can see that at 0\% corruption, the Rockafellian optimal control 
$z^{\ast}_{\rm Rock}$ for $\beta=0.9$ and $\theta =10^{-1}$ performs similarly to the optimal control $z^{\ast}_{\rm true}$ for
$\beta=0.1$ and 0\% corruption. The same is true for $z^{\ast}_{\rm Rock}$ for $\theta=10^{-2}$ and 0\% corruption
 (minimal change from the $\theta=10^{-1}$ case, and hence not shown). Notice that this does lead to poor performance
at the largest quantiles ($\ge 95$\%). 
As mentioned multiple times above, however, one can easily recover the true, risk-averse optimal control by increasing
$\theta$:  
$z^{\ast}_{\rm Rock}$ at $\theta=1$ (not shown) is identical to $z^{\ast}_{\rm true}$ at 0\% corruption.

\section{Conclusions}\label{sec:conclusions}
We introduce a framework for risk-averse optimization that is robust to ambiguities 
in the true form of the underlying problem's probability distribution. 
The framework combines problem relaxation techniques (an example of DOO) with 
the Conditional Value-at-Risk (an example of DRO), although other coherent measures of risk 
could easily be incorporated. 
The work advances previously established
theoretical foundations for risk-neutral problems with strengthened $\Gamma$-convergence
results, novel existence theorems (which can accommodate infinite-dimensional 
probability spaces), and first-order optimality criteria. 

Numerical experiments on model stochastic PDECO problems illustrate the benefit of the ``blended'' 
 DOO and DRO approach. By perturbing \textit{both} the parameter $\beta$ in 
$\on{CVaR}_{\beta}$ and the regularization parameter $\theta$ inherent to Rockafellians, 
decision makers can explore the sensitivity of optimal control strategies to changing
levels of risk aversion \textit{and} problem corruption. 
This is consistent with the notion that, on the whole, optimization technology is 
more a tool for identifying possibilities than an ``oracle'' for producing a definitive
 answer \cite[Preface]{royset2021optimization}. 

One natural extension for future research is to consider Rockafellian relaxation 
for state-constrained optimization problems in the presence of distributional 
ambiguity. Another important direction is to develop theory for how to best optimize
the bivariate objective functions that arise in Rockafellian relaxation. 
Lastly, the technique can be used in a variety of applications, for example, 
topology optimization, structural weakness identification, and trajectory
optimization.

\vskip2em

\acknowledgments{The second and fourth authors gratefully acknowledge Professors Harbir Antil and Gabriel Gatica, who made possible a visit to George Mason University, during which the core of this work was carried out. We are also thankful to Professor Johannes Royset for reviewing the manuscript and sharing his comments.}

\bibliographystyle{amsalpha}
\bibliography{references}
\end{document}